\documentclass[a4paper,11pt]{amsart}
\usepackage[latin1]{inputenc}
\usepackage{amsmath}
\usepackage{amsfonts}
\usepackage{amssymb}
\usepackage{amsthm}
\usepackage{mathrsfs}
\usepackage{graphicx}

\newcommand{\R}{{\mathbb{R}}}
\newcommand{\Q}{{\mathbb{Q}}}
\newcommand{\Z}{{\mathbb{Z}}}
\newcommand{\C}{\mathbb{C}}

\def\vecx{{\text{\boldmath$x$}}}

\def\vecz{{\text{\boldmath$z$}}}

\def\vecv{{\text{\boldmath$v$}}}
\def\vecm{{\text{\boldmath$m$}}}

\def\vec0{{\text{\boldmath$0$}}}

\def\Re{\operatorname{Re}}

\newcommand{\ve}{\varepsilon}

\newcommand{\sfrac}[2]{{\textstyle \frac {#1}{#2}}}

\newcommand{\SL}{\mathrm{SL}}

\newtheorem{thm}{Theorem}[section]
\newtheorem{lem}[thm]{Lemma}
\newtheorem{prop}[thm]{Proposition}
\newtheorem{cor}[thm]{Corollary}

\theoremstyle{remark}
\newtheorem{remark}[thm]{Remark}

\numberwithin{equation}{section}

\begin{document}
\title[On the value distribution of the Epstein zeta function]{On the value distribution of the Epstein zeta function in the critical strip}
\author{Anders S\"odergren}
\address{Department of Mathematics, Box 480, Uppsala University, 751 06 Uppsala, Sweden\newline
\rule[0ex]{0ex}{0ex}\hspace{8pt} {\tt sodergren@math.uu.se}\newline
\rule[0ex]{0ex}{0ex} \hspace{8pt}\textit{Present address:} 
School of Mathematics, Institute for Advanced Study, Einstein\newline 
\rule[0ex]{0ex}{0ex} \hspace{8pt}Drive, Princeton, NJ 08540, USA\newline
\rule[0ex]{0ex}{0ex} \hspace{8pt}{\tt sodergren@math.ias.edu}}  
\date{\today}
\thanks{This material is based upon work supported in part by the Swedish Research Council, Research Grant 621-2007-6352, and in part by the National Science Foundation under agreement No.\ DMS-0635607. Any opinions, findings and conclusions or recommendations expressed in this material are those of the author and do not necessarily reflect the views of the National Science Foundation.}

\maketitle

\begin{abstract}
We study the value distribution of the Epstein zeta function $E_n(L,s)$ for $0<s<\frac{n}{2}$ and a random lattice $L$ of large dimension $n$. For any fixed $c\in(\frac{1}{4},\frac{1}{2})$ and $n\to\infty$, we prove that the random variable $V_n^{-2c}E_n(\cdot,cn)$ has a limit distribution, which we give explicitly (here $V_n$ is the volume of the $n$-dimensional unit ball). More generally, for any fixed $\ve>0$ we determine the limit distribution of the random function $c\mapsto V_n^{-2c}E_n(\cdot,cn)$, $c\in[\frac{1}{4}+\ve,\frac{1}{2}-\ve]$. After compensating for the pole at $c=\frac12$ we even obtain a limit result on the whole interval $[\frac14+\ve,\frac12]$, and as a special case we deduce the following strengthening of a result by Sarnak and Str\"ombergsson \cite{sst} concerning the height function $h_n(L)$ of the flat torus $\R^n/L$: The random variable $n\bigl\{h_n(L)-(\log(4\pi)-\gamma+1)\bigr\}+\log n$ has a limit distribution as $n\to\infty$, which we give explicitly. Finally we discuss a question posed by Sarnak and Str\"ombergsson as to whether there exists a lattice $L\subset\R^n$ for which $E_n(L,s)$ has no zeros in $(0,\infty)$.
\end{abstract}

\section{Introduction}

Let $X_n$ denote the space of $n$-dimensional lattices of covolume $1$. We realize $X_n$ as the homogeneous space $\SL(n,\Z)\backslash \SL(n,\R)$, where $\SL(n,\Z)g$ corresponds to the lattice $\Z^ng\subset\R^n$. We further let $\mu_n$ denote the Haar measure on $\SL(n,\R)$, normalized to be the unique right $\SL(n,\R)$-invariant probability measure on $X_n$.

For $L\in X_n$ and $\Re s>\frac{n}{2}$ the Epstein zeta function is defined by
\begin{align*}
E_n(L,s)={\sum_{\vecm\in L}}'|\vecm|^{-2s},
\end{align*}
where $'$ denotes that the zero vector should be omitted. $E_n(L,s)$ has an analytic continuation to $\C$ except for a simple pole at $s=\frac{n}{2}$ with residue $\pi^{\frac n2}\Gamma(\frac{n}{2})^{-1}$. Furthermore $E_n(L,s)$ satisfies the functional equation
\begin{align}\label{functionaleq}
F_n(L,s)=F_n(L^*,\sfrac{n}{2}-s),
\end{align}
where
\begin{align}\label{Fn}
F_n(L,s):=\pi^{-s}\Gamma(s)E_n(L,s),
\end{align}
and $L^*$ is the dual lattice of $L$. The close relation with the Riemann zeta function, in fact $ \zeta(2s)=\frac{1}{2}E_1(\Z,s)$, makes it natural to call the region $0<\Re s<\frac{n}{2}$ the critical strip for $E_n(L,s)$. Note however that for all $n\geq2$ there exist lattices $L\in X_n$ for which the Riemann hypothesis for $E_n(L,s)$ is known to fail (cf. \cite[Thm.\ 1]{terras3}; see also \cite{bateman}, \cite{stark}, \cite{terras2} and \cite{terras}).

It follows from \eqref{functionaleq} that $E_n(L,0)=-1$ for all $L\in X_n$. Since $E_n(L,s)$ has a simple pole at $s=\frac{n}{2}$ with positive residue it is also clear that 
\begin{align*}
\lim_{s\to\frac n2-}E_n(L,s)=-\infty
\end{align*}
for all $L\in X_n$. In this paper we will be interested in the behavior of $E_n(L,s)$  in the interval $0<s<\frac{n}{2}$ for large $n$. In particular we will, for $0<c<\frac{1}{2}$, be interested in questions concerning the value distribution of $E_n(L,cn)$ as $n\to\infty$. These questions are mainly motivated by the work of Sarnak and Strömbergsson \cite{sst} on minima of $E_n(L,s)$. They note that if there exists a lattice $L_0\in X_n$ satisfying $E_n(L,s)\geq E_n(L_0,s)$ for all $0<s<\frac{n}{2}$ and all $L\in X_n$ then $E_n(L_0,s)<0$ for $0<s<\frac{n}{2}$. Hence, for such a lattice $L_0$, $E_n(L_0,s)$ has no zeros in $(0,\infty)$. 

The question as to whether or not a lattice with the last property can exist is also of interest in algebraic number theory. In particular, by Hecke's integral formula (cf.\ \cite[pp.\ 198-207]{hecke} and \cite[eq.\ (9)]{terras}), if we knew that $E_n(L,s)<0$ for all $0<s<\frac{n}{2}$ and all lattices $L\in X_n$ of a special type related to a given number field $k$, this would imply that the Dedekind zeta function $\zeta_k(s)$ of $k$ satisfies $\zeta_k(s)<0$ for all $s\in(0,1)$!

Gaining insight into whether or not lattices $L\in X_n$ with $E_n(L,s)\neq0$, $\forall s>0$, do exist for all $n$ (or all large $n$) is one of the main goals of the present study. A first step in this direction was taken by Sarnak and Str\"ombergsson in \cite[Sec.\ 6]{sst}, where they study the value distribution of the height function for flat tori as $n\to\infty$. Recall that for the flat torus $\R^n/L$, with $L\in X_n$, the height function is given by
\begin{align}\label{heightdef}
 h_n(\R^n/L)=h_n(L)=2\log(2\pi)+\frac{\partial}{\partial s}E_n(L^*,s)_{|s=0}.
\end{align}
Theorem 3 of \cite{sst} states that if $\ve>0$ is fixed then
\begin{align}\label{heightresult}
\text{Prob}_{\mu_n}\Big\{L\in X_n\,\,\Big|\,\,\big|h_n(L)-(\log(4\pi)-\gamma+1)\big|<\ve\Big\}\to 1
\end{align}
as $n\to\infty$, where $\gamma$ is Euler's constant. Expressed in terms of the Epstein zeta function, \eqref{heightresult} says
\begin{align}\label{HRE}
\text{Prob}_{\mu_n}\Big\{L\in X_n\,\,\Big|\,\,\big|\sfrac{\partial}{\partial s}E_n(L,s)_{|s=0}-(1-\gamma-\log \pi)\big|<\ve\Big\}\to 1
\end{align}
as $n\to\infty$. Here $1-\gamma-\log(\pi)\approx-0.72$. Note that \eqref{HRE} together with $E_n(L,0)=-1$ ($\forall L\in X_n$) give a fairly precise description of the behavior of $E_n(L,s)$ in the left end of the interval $0<s<\frac{n}{2}$ for most $L\in X_n$  when $n$ is large.

The results in the present paper give information on the value distribution of $E_n(L,s)$ for $\frac{n}{4}<s<\frac{n}{2}$ with large $n$. Using \eqref{functionaleq} it is then easy to infer results also for the interval $0<s<\frac{n}{4}$. In order to state our theorems we first need to introduce some notation. We consider a Poisson process $\mathcal P=\big\{\widehat N(V),V\geq0\big\}$ on the positive real line with constant intensity $\frac{1}{2}$, and let   $T_1, T_2,T_3,\ldots$ denote the points of the process ordered in such a way that $0<T_1<T_2< T_3<\ldots$. We let
$N(V):=2\widehat N(V)$ and define, for all $V\geq0$,
\begin{align}\label{Rvdef}
R(V):=N(V)-V.
\end{align}
Finally we let $V_n$ denote the volume of the unit ball in $\R^n$.

\begin{thm}\label{curvethm}
Let $\frac{1}{4}<c_1<c_2<\frac{1}{2}$. For each $n\in\Z_{\geq1}$ consider
\begin{align*}
 c\mapsto V_n^{-2c}E_n(\cdot,cn)
\end{align*}
as a random function in $C\big([c_1,c_2]\big)$. Then the distribution of this random function converges to the distribution of
\begin{align*}
 c\mapsto \int_{0}^{\infty}V^{-2c}\,dR(V)
\end{align*}
as $n\to\infty$.
\end{thm} 

For our purposes it is essential to understand $V_n^{-2c}E_n(\cdot,cn)$ as a random function. Nevertheless, for extra clarity we also state the following immediate corollary of Theorem \ref{curvethm}. 

\begin{cor}\label{firstthm}
For fixed $c\in(\frac{1}{4},\frac{1}{2})$, the distribution of the random variable $V_n^{-2c}E_n(\cdot,cn)$ converges to the distribution of $\int_{0}^{\infty}V^{-2c}\,dR(V)$ as $n\to\infty$. In fact, for any $m\geq1$ and fixed $\frac{1}{4}<c_1<\cdots<c_m<\frac{1}{2}$, the distribution of the random vector
\begin{align*}
\Big(V_n^{-2c_1}E_n(\cdot,c_1n),\ldots,V_n^{-2c_m}E_n(\cdot,c_mn)\Big)
\end{align*}
converges to the distribution of
\begin{align*}
\bigg(\int_{0}^{\infty}V^{-2c_1}\,dR(V),\ldots,\int_{0}^{\infty}V^{-2c_m}\,dR(V)\bigg)
\end{align*}
as $n\to\infty$.
\end{cor} 

The fact that the limit random variables in Theorem \ref{curvethm} and Corollary \ref{firstthm} are well-defined follows from the bound 
\begin{align}\label{prelest}
|R(V)|\ll (V\log\log V)^{\frac12}\qquad \text{ as } V\to\infty,
\end{align}
which holds almost surely, as a simple consequence of the law of the iterated logarithm. We also mention that the distribution of $\int_{0}^{\infty}V^{-2c}\,dR(V)$, for fixed $c\in(\frac14,\frac12)$, is well understood. In particular $\int_{0}^{\infty}V^{-2c}\,dR(V)$ has a strictly $\frac1{2c}$-stable distribution. We discuss these matters in detail in Section \ref{HC*}.

Let us point out the close formal similarity between the results above and our previous results in \cite{epstein1} on the value distribution of $E_n(\cdot,cn)$ to the right of the critical strip. In the language we have adopted here the main result in \cite{epstein1} states that for fixed $c>\frac{1}{2}$, the distribution of the random variable $V_n^{-2c}E_n(\cdot,cn)$ converges to the distribution of $\int_{0}^{\infty}V^{-2c}\,dN(V)$ as $n\to\infty$. Similar statements also hold for general finite dimensional distributions and the corresponding random functions. Hence, passing from the case to the right of the critical strip to the present one, we need only change from "$dN(V)$" to "$dR(V)$" in the limit variable.

A crucial ingredient in the proof of Theorem \ref{curvethm} is our result \cite{poisson} on the distribution of lengths of lattice vectors in a random lattice $L\in X_n$. It says that, as $n\to\infty$, the suitably normalized non-zero vector lengths in a random lattice $L\in X_n$ behave like the points of a Poisson process on the the positive real line. To be more precise: Given a lattice $L\in X_n$, we order its non-zero vectors by increasing lengths as $\pm\vecv_1,\pm\vecv_2,\pm\vecv_3,\ldots$, set $\ell_j=|\vecv_j|$ (thus $0<\ell_1\leq \ell_2\leq \ell_3\leq\ldots$), and define
\begin{align}\label{volumes}
 \mathcal V_j:=\frac{\pi^{n/2}}{\Gamma(\frac{n}{2}+1)}\ell_j^n\,,
\end{align}
so that $\mathcal V_j$ is the volume of an $n$-dimensional ball of radius $\ell_j$. The main result in \cite{poisson} now states that, as $n\to\infty$, the volumes $\{\mathcal V_j\}_{j=1}^{\infty}$ determined by a random lattice $L\in X_n$ converges in distribution to the points $\{T_j\}_{j=1}^{\infty}$ of the  Poisson process $\mathcal P$ on the positive real line with constant intensity $\frac{1}{2}$.

In view of this result from \cite{poisson}, the following definitions are natural. Given $L\in X_n$ and $V\geq0$ we let $N_n(V)$ denote the number of non-zero lattice points of $L$ in the closed $n$-ball of volume $V$ centered at the origin, and  define
\begin{align}\label{Rnvdef}
R_n(V):=N_n(V)-V.
\end{align}
Note that the above-mentioned result from \cite{poisson} implies in particular that $N_n(V)$ tends in distribution to $N(V)$ as $n\to\infty$, and $R_n(V)$ tends in distribution to $R(V)$, for any $V\geq0$.

A second crucial ingredient in our proof of Theorem \ref{curvethm} is a bound of similar quality as \eqref{prelest} for the corresponding function $R_n(V)$ on $X_n$.

\begin{thm}\label{Rnthm}
For all $\ve>0$ there exists $C_{\ve}>0$ such that for all $n\geq3$ and $C\geq1$ we have 
\begin{align*}
\text{Prob}_{\mu_n}\Big\{L\in X_n\,\,\big|\,\,|R_n(V)|\leq C_{\ve}(CV)^{\frac{1}{2}}(\log V)^{\frac{3}{2}+\ve},\quad \forall V\geq10\Big\}\geq1-C^{-1}.
\end{align*}
\end{thm}

We stress in particular that $C_{\ve}$ is independent of $n$.

Theorem \ref{Rnthm} is interesting not only for being an important technical part of the proof of Theorem \ref{curvethm}, but also for its connection with the famous circle problem generalized to dimension $n$ and general ellipsoids. Given $V>0$, $n\geq2$ and $L\in X_n$ the problem asks for the number $\mathcal N(V)=1+N_n(V)$ of  lattice points of $L$ in the closed $n$-ball of volume $V$ centered at the origin. It is well-known that $\mathcal N(V)$ is asymptotic to the volume $V$ of this ball. Hence $1+R_n(V)$ equals the remainder term in this asymptotic relation, and Theorem \ref{Rnthm} implies that this remainder is $\ll V^{\frac12}(\log V)^{\frac{3}{2}+\ve}$ as $V\to\infty$, for almost every $L\in X_n$.

As far as we are aware, the fact that almost every $L\in X_n$ satisfies $|R_n(V)|\ll V^{\frac12}(\log V)^{\frac32+\ve}$, or just $|R_n(V)|\ll V^{\frac12+\ve}$, as $V\to\infty$, has not been pointed out previously in the literature. We mention a result from 1928 by Jarnik \cite[Satz 3]{jarnik}, which in our notation says that $|R_n(V)|\ll V^{\frac12+\ve}$ holds for almost every orthogonal lattice $L$ (viz.\ a lattice which has an orthogonal $\Z$-basis), 
when $n\geq4$. Also in this vein we mention the impressive recent work by Bentkus and G\"otze \cite{BG1}, \cite{BG2} and G\"otze \cite{gotze}, which imply strong explicit bounds on $R_n(V)$ for an arbitrary given lattice $L$. In particular, \cite{gotze} implies that $|R_n(V)|\ll V^{1-\frac2n}$ holds for every $L\in X_n$ when $n\geq5$, and furthermore the stronger bound $R_n(V)=o(V^{1-\frac2n})$ as $V\to\infty$ whenever $L$ is irrational in the sense that the Gram matrix for some $\Z$-basis of $L$ (equivalently: for every $\Z$-basis of $L$) is not proportional to a matrix with integer entries only.

\vspace{6pt}

In Section \ref{heightsec} we extend the result in Theorem \ref{curvethm} to the case $c_2=\frac12$. In order for this to make sense we have to subtract the singular part of $V_n^{-2c}E_n(\cdot,cn)$ from both the random functions appearing in Theorem \ref{curvethm}. A precise statement of this limit value distribution result can be found in Theorem \ref{curvethm2}. As an application we prove a result on the asymptotic value distribution of the height function $h_n$. First, in Lemma \ref{Z0lemma}, we show that the limit 
\begin{align}\label{Z0intro}
Z_0:=\lim_{c\to\frac{1}{2}-}\bigg(\int_0^{\infty}V^{-2c}\,dR(V)+\frac{1}{1-2c}\bigg)
\end{align}
exists almost surely. Recall that it was proved in \cite[Thm.\ 3]{sst} that the random variable $h_n(L)$ converges in distribution to the constant $\log(4\pi)-\gamma+1$ (cf.\ \eqref{heightresult} above). Relating $Z_0$ to a similar limit involving $E_n(L,cn)$ and using the functional equation \eqref{functionaleq} and the formula \eqref{heightdef} for $h_n$, we obtain the following much more precise convergence result: 

\begin{thm}\label{myheight}
The random variable 
\begin{align*}
n\Big(h_n(L)-\big(\log(4\pi)-\gamma+1\big)\Big)+\log n
\end{align*}
converges in distribution to 
\begin{align*}
2Z_0-\log\pi-1
\end{align*}
as $n\to\infty$.
\end{thm}

Returning to the question of whether there exists a lattice $L\in X_n$ such that $E_n(L,s)<0$ for $0<s<\frac{n}{2}$, we note that Theorem \ref{curvethm} and Theorem \ref{curvethm2} have the following corollary.

\begin{cor}\label{negativity}
For any fixed $\frac14<c_1<c_2\leq\frac{1}{2}$, the limit 
\begin{align*}
\lim_{n\to\infty}\text{Prob}_{\mu_n}\Big\{L\in X_n\,\big|\, E_n(L,s)<0 \text{  for all  } s\in[c_1n,c_2n]\setminus\{\sfrac12n\}\Big\}
\end{align*} 
exists, and equals
\begin{align*}
f(c_1,c_2):=\text{Prob}\bigg\{\int_{0}^{\infty}V^{-2c}\,dR(V)<0 \text{  for all  } c\in[c_1,c_2]\setminus\{\sfrac12\}\bigg\}.
\end{align*}
Moreover, for all $\frac14<c_1<c_2\leq\frac{1}{2}$ the probability $f(c_1,c_2)$ satisfies $0<f(c_1,c_2)<1$.
\end{cor}

In particular, for any given $\ve>0$ the probability that
\begin{align*}
E_n(L,s)<0\quad\text{for all}\quad s\in\big[(\sfrac14+\ve)n,\sfrac12n\big)
\end{align*}
holds tends to a positive limit as $n\to\infty$! However, we also have the following results.

\begin{thm}\label{independentgaussians}
Fix $m\in\Z_{\geq1}$ and let $c_j=\frac14+\eta_j$ with $\eta_j\in(0,\frac14)$ for $1\leq j\leq m$. If $(\eta_1,\ldots,\eta_m)$ tends to the zero vector in $\R^m$ in such a way that $\eta_j/\eta_{j+1}\to0$ for each $1\leq j\leq m-1$, then the $m$-dimensional random vector 
\begin{align*}
\bigg(\big(2c_1-\sfrac12\big)^{\frac12}\int_{0}^{\infty}V^{-2c_1}\,dR(V),\ldots,\big(2c_m-\sfrac12\big)^{\frac12}\int_{0}^{\infty}V^{-2c_m}\,dR(V)\bigg)
\end{align*}
converges in distribution to the distribution of $m$ independent $N(0,1)$-variables.
\end{thm}

\begin{cor}
For each fixed $c_2\in\big(\frac14,\frac12\big]$, the probability $f(c_1,c_2)$ tends to zero as $c_1\to\frac14+$. 
\end{cor}

As an immediate consequence it follows that for any $\ve>0$ we have
\begin{align*}
\text{Prob}_{\mu_n}\Bigl\{L\in X_n\:\big|\:
E_n(L,s)<0\text{ for all }s\in\bigl[\sfrac14n,(\sfrac14+\ve)n\bigr]\Bigr\}\to0
\end{align*}
as $n\to\infty$. In particular this entails that the probability that $E_n(L,s)$ has a zero in $(0,\infty)$ tends to one as $n\to\infty$. Hence the question of Sarnak and Str\"ombergsson is rather delicate!

Finally we remark that the precise behavior of the random variable $E_n(L,cn)$ for $c=\frac14$ or $c$ tending to $\frac14$ as $n\to\infty$ remains very much an open and exciting question, which we hope to tackle in future work.

\section{The random variables $H(c)$ and $Z_0$}\label{HC*}

\subsection{The random variable $H(c)$}\label{HC}

In this section we prove some basic results about the random variable 
\begin{align}\label{Hcdef}
H(c):=\int_{0}^{\infty}V^{-2c}\,dR(V),
\end{align}
which appears as the limit variable in Theorem \ref{curvethm} and Corollary \ref{firstthm}.

Recall from the introduction that, for a Poisson process $\mathcal P=\big\{\widehat N(V),V\geq0\big\}$ on the positive real line with constant intensity $\frac{1}{2}$, we let $N(V):=2\widehat N(V)$ and define
\begin{align*}
R(V):=N(V)-V,\qquad V\geq0.
\end{align*}
We also recall that $\widehat N(V)$ denotes the number of points of $\mathcal P$ falling in the interval $(0,V]$ and that $\widehat N(V)$ is Poisson distributed with expectation value $\frac{1}{2}V$. In fact, since furthermore $\widehat N(V_2)-\widehat N(V_1)$ is Poisson distributed with expectation value $\frac{1}{2}(V_2-V_1)$, it follows that $\mathbb E\big(R(V_2)-R(V_1)\big)=0$ and 
\begin{align}\label{variance}
\mathbb E\Big(\big(R(V_2)-R(V_1)\big)^2\Big)=\text{Var}\big(R(V_2)-R(V_1)\big)=\text{Var}\big(N(V_2)-N(V_1)\big)=2(V_2-V_1)
\end{align} 
for all $0\leq V_1<V_2$. We let $T_1, T_2,T_3,\ldots$ denote the points of $\mathcal P$ ordered in such a way that $0<T_1<T_2< T_3<\ldots$. Hence the sequence $\{T_j\}_{j=1}^{\infty}$ belongs to the space 
\begin{align*}
\Omega:=\Big\{\vecx=\{x_j\}_{j=1}^{\infty}\in(\R_{\geq0})^{\infty}\,\big|\, 0<x_1<x_2<x_3<\ldots
\Big\}. 
\end{align*}
We equip $\Omega$ with the subspace topology induced from the product topology on $(\R_{\geq0})^{\infty}$. We denote the distribution of $\mathcal{P}$ on $\Omega$ by $\textbf{P}$ and note that $\textbf{P}$ is actually a Borel probability measure on $\Omega$.

To begin with we need an estimate of $R(V)$. Using the law of the iterated logarithm (see \cite{HW}) it is straightforward to show that with probability one we have
\begin{align*}
\limsup_{V\to\infty}\frac{|R(V)|}{(V\log\log V)^{\frac12}}=2.
\end{align*} 
In particular it follows that with probability one there exists a constant $C>2$ (that depends on $\vecx\in\Omega$) such that
\begin{align}\label{Riteratedlog}
|R(V)|<C(V\log\log V)^{\frac12}, \qquad \forall V\geq10.
\end{align}

In the following lemma we give a simple proof of a slightly weaker bound than \eqref{Riteratedlog}, which as input only uses the monotonicity of $N(V)$ and the variance relation \eqref{variance}. This proof has the advantage that it easily generalizes to the situation in Theorem \ref{Rnthm} (see Section \ref{Rnsec}).

\begin{lem}\label{Rest}
For all $\ve>0$ there exists $C_{\ve}>0$ such that for all $C\geq1$ we have 
\begin{align*}
\textbf{P}\,\Big\{|R(V)|\leq C_{\ve}(CV)^{\frac{1}{2}}(\log V)^{\frac{3}{2}+\ve},\quad \forall V\geq10\Big\}\geq1-C^{-1}.
\end{align*}
\end{lem}

\begin{remark}
Note that the set 
\begin{align*}
\Bigl\{\vecx\in\Omega\: \big|\: |R(V)|\leq C_\ve(CV)^{\frac12}(\log V)^{\frac32+\ve},\:\forall V\geq10\Bigr\}
\end{align*}
is indeed $\mathbf P$-measurable, viz.\ a Borel subset of $\Omega$. Indeed, since $R(V)$ is right-continuous for every $\vecx\in\Omega$, the above set equals the countable intersection
\begin{align*}
\bigcap_{V\in\Q\cap[10,\infty)}\Bigl\{\vecx\in\Omega\:\big|\: |R(V)|\leq C_\ve(CV)^{\frac12}(\log V)^{\frac32+\ve}\Bigr\}.
\end{align*}
Here each set is of the form $\big\{\vecx\in\Omega\mid |R(V)|\leq A\big\}$ for some $V,A\geq0$, and since
\begin{align*}
\big\{\vecx\in\Omega\mid |R(V)|\leq A\big\}=\big\{\vecx\in\Omega\mid x_m\leq V\text{ and }x_{\ell+1}>V\big\}
\end{align*}
with $m=\lceil\frac12(V-A)\rceil$ and $\ell=\lfloor\frac12(V+A)\rfloor$, this is a Borel subset of $\Omega$. In a similar way one also proves that the set $\Omega_\ve$ defined below in \eqref{omegaepsilon} is a Borel subset of $\Omega$, and also that the set considered in Theorem \ref{Rnthm} is a Borel subset of $X_n$.
\end{remark}

\begin{proof}[Proof of Lemma \ref{Rest}]
For all $A\geq10$, it follows from \eqref{variance} that 
\begin{align*}
\mathbb E\bigg(R(A)^2+\sum_{0\leq k\leq\frac{1}{2}\log_2A}\sum_{j=0}^{2^{k}-1}\Big(R\big((1+2^{-k}(j+1))A\big)-R\big((1+2^{-k}j)A\big)\Big)^2\bigg)\\
=2A+\sum_{0\leq k\leq\frac{1}{2}\log_2A}2^k\cdot2\cdot2^{-k}A=2A\big(\lfloor\sfrac{1}{2}\log_2A\rfloor+2\big)\ll A\log A,
\end{align*}
where the implied constant is absolute. Hence, using Markov's inequality, we get
\begin{multline}\label{markov}
\textbf P\bigg\{R(A)^2+\sum_{0\leq k\leq\frac{1}{2}\log_2A}\sum_{j=0}^{2^{k}-1}\Big(R\big((1+2^{-k}(j+1))A\big)-R\big((1+2^{-k}j)A\big)\Big)^2\\
\geq CA\log A\bigg\}\ll C^{-1},
\end{multline}
uniformly over all $C>0$ and $A\geq10$. On the other hand we claim that for all  $C\geq1$, $A\geq10$ and $\vecx\in\Omega$ for which 
\begin{align}\label{condition}
R(A)^2+\sum_{0\leq k\leq\frac{1}{2}\log_2A}\sum_{j=0}^{2^{k}-1}\Big(R\big((1+2^{-k}(j+1))A\big)-R\big((1+2^{-k}j)A\big)\Big)^2<CA\log A
\end{align}
holds, we have 
\begin{align}\label{claim}
|R(V)|\ll(CA)^{\frac{1}{2}}\log A\qquad\text{for all $V\in[A,2A]$},
\end{align}
with an absolute implied constant. 

To prove the claim we fix any $V\in[A,2A]$. We also set $k_0:=\lfloor\frac{1}{2}\log_2A\rfloor$ and let $m$ be the largest integer satisfying $(1+2^{-k_0}m)A\leq V$; thus $0\leq m\leq 2^{k_0}$. By considering the binary representation of $m$, we may express 
\begin{align*}
R\big((1+2^{-k_0}m)A\big)-R(A)
\end{align*}
as a sum of terms of the form
\begin{align*}
R\big((1+2^{-k}(j+1))A\big)-R\big((1+2^{-k}j)A\big),
\end{align*}
where $0\leq k\leq k_0$ and where for each $k\in\{0,\ldots,k_0\}$, we either have no term, or exactly one term, for some $j=j(k,m)\in\{0,\ldots,2^k-1\}$. Hence the total number of terms does not exceed $k_0+1$ and by the Cauchy-Schwarz inequality and \eqref{condition} we have
\begin{align*}
&\big|R\big((1+2^{-k_0}m)A\big)-R(A)\big|\\
&\leq(k_0+1)^{\frac12}\bigg(\sum_{0\leq k\leq\frac{1}{2}\log_2A}\sum_{j=0}^{2^{k}-1}\Big(R\big((1+2^{-k}(j+1))A\big)-R\big((1+2^{-k}j)A\big)\Big)^2\bigg)^{\frac12}\\
&<(k_0+1)^{\frac12}\big(CA\log A\big)^{\frac12}\ll(CA)^{\frac{1}{2}}\log A.
\end{align*}
Using \eqref{condition} once more we get $R(A)<\big(CA\log A\big)^{\frac12}$ and thus, by the triangle inequality, we conclude that
\begin{align}\label{firstest}
\big|R\big((1+2^{-k_0}m)A\big)\big|\ll(CA)^{\frac{1}{2}}\log A.
\end{align}
Now, if $V=2A$ then $m=2^{k_0}$, $R(V)=R\big((1+2^{-k_0}m)A\big)$ and \eqref{firstest} is the desired estimate. Next we assume that $V<2A$. Then $m+1\leq2^{k_0}$ and by the argument proving \eqref{firstest} we also get
\begin{align}\label{secondest}
\big|R\big((1+2^{-k_0}(m+1))A\big)\big|\ll(CA)^{\frac{1}{2}}\log A.
\end{align}
Using the definition of $R(X)$ and the fact that $N(X)$ is an increasing function of $X$ we obtain 
\begin{align*}
R(X)\leq R(X')+X'-X\qquad \text{ for all } 0\leq X\leq X'.
\end{align*}
Thus, since we by our choice of $m$ have 
\begin{align*}
(1+2^{-k_0}m)A\leq V<(1+2^{-k_0}(m+1))A,
\end{align*}
we get
\begin{align}\label{thirdest}
R\big((1+2^{-k_0}m)A\big)-2^{-k_0}A\leq R(V)\leq R\big((1+2^{-k_0}(m+1))A\big)+2^{-k_0}A.
\end{align}
Recalling that $k_0>\frac{1}{2}\log_2A-1$ we obtain $2^{-k_0}A<2A^{\frac12}$. Hence \eqref{firstest}, \eqref{secondest} and \eqref{thirdest} together conclude the proof of the claim that \eqref{condition} implies \eqref{claim}.

Combining \eqref{markov} with the fact that \eqref{condition} implies \eqref{claim}, yields the following statement: There exists an absolute constant $C_0>0$ such that for all $C\geq1$ and $A\geq10$ we have 
\begin{align}\label{prob1}
\textbf{P}\Big\{\exists V\in[A,2A] : |R(V)|>C_0(CA)^{\frac12}\log A\Big\}\leq C^{-1}.
\end{align}
(Note that the constant $C$ in \eqref{prob1} is an appropriate multiple of the constant $C$ in \eqref{markov}-\eqref{claim}.) Now, given $K\geq1$ and $\ve>0$, we apply, for all $j\in\Z_{\geq1}$, \eqref{prob1} with $A=5\cdot2^j$ and $C=Kj^{1+\ve}$. We conclude that there exists a constant $C_0'>0$, which only depends on $\ve$, such that
\begin{align*}
\textbf{P}\Big\{\exists V\in[5\cdot2^j,10\cdot2^j] : |R(V)|>C_0'(KV)^{\frac12}(\log V)^{\frac32+\frac12 \ve}\Big\}\leq K^{-1}j^{-1-\ve}
\end{align*}
for all $j\in\Z_{\geq1}$. Hence, using the subadditivity of $\textbf P$, we obtain 
\begin{align*}
\textbf{P}\Big\{\exists V\geq10: |R(V)|>C_0'(KV)^{\frac12}(\log V)^{\frac32+\frac12 \ve}\Big\}\leq K^{-1}C',
\end{align*}
where $C':=\sum_{j=1}^{\infty}j^{-1-\ve}>1$ only depends on $\ve$. Finally, the lemma follows from setting $C=K{C'}^{-1}$ and $C_{\ve/2}=C_0'{C'}^{\frac12}$.
\end{proof}

For $\ve>0$ we define
\begin{align}\label{omegaepsilon}
\Omega_{\ve}:=\Big\{\vecx\in\Omega\,\,\big|\,\,|R(V)|\ll_{\vecx,\ve}V^{\frac12}(\log V)^{\frac32+\ve} \quad \forall V\geq10\Big\}.
\end{align}
Note that it follows from Lemma \ref{Rest} that $\textbf P(\Omega_{\ve})=1$ for every $\ve>0$. For notational convenience we will only work with $\Omega_{1/2}$ in the following; however any other set $\Omega_\ve$ would do just as well. The following lemma shows that the integral $H(c)$ in \eqref{Hcdef} converges almost surely.

\begin{lem}\label{lordagslemma}
For every $\vecx\in\Omega_{1/2}$ the integral $H(c)$ converges for all $c\in(\frac14,\frac12)$, and furthermore the integral $\int_A^\infty V^{-2c}\,dR(V)$ converges for all $A>0$ and $c>\frac14$.
\end{lem}

\begin{proof}
Let $\vecx\in\Omega_{1/2}$ be fixed. Now for any $0<A<B$ and $c>\frac14$ we have
\begin{align}\label{ABINTEGRAL2}
\int_A^B V^{-2c}\,dR(V)=\Bigl[V^{-2c}R(V)\Bigr]_{V=A}^{V=B}+2c\int_A^B V^{-2c-1}R(V)\,dV,
\end{align}
and since $V^{-2c}|R(V)|\ll_\vecx V^{\frac12-2c}(\log V)^2$ as $V\to\infty$, with $\frac12-2c<0$,
it follows that both terms in the right hand side of \eqref{ABINTEGRAL2} are convergent as $B\to\infty$.
This proves the second statement of the lemma. Finally, since $R(V)=-V$ for all $0\leq V\ll_\vecx 1$ it follows that if $c<\frac12$ then the two terms in the right hand side of \eqref{ABINTEGRAL2} are also convergent as $A\to0$, so that $H(c)$ converges for all $c\in(\frac14,\frac12)$.
\end{proof}

\begin{lem}\label{welldef1}
$H(c)$ is a well-defined random variable on $\Omega_{1/2}$ for all $c\in(\frac{1}{4},\frac{1}{2})$.
\end{lem}

\begin{proof}
Fix $c\in(\frac{1}{4},\frac{1}{2})$. By Lemma \ref{lordagslemma}, $H(c)$ is convergent for each $\vecx\in\Omega_{1/2}$, and it remains to show that  $\vecx\mapsto H(c)$ is measurable. Let us for $A>0$ consider the function $f_A:\Omega\to\R\cup\{\infty\}$ defined by 
\begin{align}\label{FA}
f_A(T_1,T_2,\ldots)=\int_0^AV^{-2c}\,dR(V)=\int_0^AV^{-2c}\,dN(V)-\int_0^AV^{-2c}\,dV\\
=2\sum_{T_j\leq A}T_j^{-2c}-\frac{A^{1-2c}}{1-2c}\,.\nonumber
\end{align}
We express $\Omega$ as a disjoint union of Borel sets as follows: $\Omega=\big(\cup_{j=0}^{\infty}\Omega^{(j)}\big)\cup\Omega^{(\infty)}$, where 
\begin{align}\label{mangd1}
\Omega^{(\infty)}=\big\{\vecx\in\Omega\mid x_{\ell}\leq A,\,\forall\ell\big\},\qquad\Omega^{(0)}=\big\{\vecx\in\Omega\mid A<x_1\big\},
\end{align} 
and
\begin{align}\label{mangd2}
\Omega^{(j)}=\big\{\vecx\in\Omega\mid x_j\leq A<x_{j+1}\big\}\,\text{ for $j\geq1$.}
\end{align} 
It follows from the last expression in \eqref{FA} that the restriction of $f_A$ to each set $\Omega^{(j)}$ is continuous (we set $f_A:=\infty$ for all $\vecx\in\Omega^{(\infty)}$). Hence each $f_A$ is  measurable, and hence also the restrictions of these functions to $\Omega_{1/2}$ are measurable (of course we also have $\Omega^{(\infty)}\cap\Omega_{1/2}=\emptyset$, so that $f_A$ is real-valued on $\Omega_{1/2}$). Thus also $H(c)$ is measurable on $\Omega_{1/2}$, since it is the pointwise limit of the sequence $f_1,f_2,f_3,\ldots$ of measurable functions.   
\end{proof}

\begin{remark}\label{welldef2}
We want to consider $H(c)$ also as a random variable on $\Omega$. To make this rigorous we should redefine $H(c)$ (as for example zero) on $\Omega\setminus\Omega_{1/2}$ in order to make $H(c)$ measurable on $\Omega$ (cf.\ \cite[p.\ 29]{rudin}). However, since we in the present paper are only interested in questions of distribution and $\Omega_{1/2}$ has full measure in $\Omega$, we will simply let $H(c)$ remain undefined at points where the integral is divergent. 
\end{remark}

We next note that Lemma \ref{Rest} also implies that the tail of $H(c)$ can be made uniformly small in closed intervals $[c_1,c_2]\subset(\frac14,\frac12]$.

\begin{lem}\label{intest}
Let $\frac{1}{4}<c_1<c_2\leq\frac{1}{2}$. Then for all $\ve'>0$ there exists a constant $A_0>0$ such that for all $A\geq A_0$ we have
\begin{align*}
\textbf{P}\,\bigg\{\sup_{c\in[c_1,c_2]}\bigg|\int_A^{\infty}V^{-2c}\,dR(V)\bigg|\leq\ve'\bigg\}\geq1-\ve'.
\end{align*}
\end{lem}

\begin{proof}
Let $\ve'>0$ and $\delta\in(0,2c_1-\frac12)$ be given. It follows from Lemma \ref{Rest} that there exists a set $\Omega'\subset\Omega_{1/2}$ with $\textbf P (\Omega')\geq 1-\ve'$ such that for all $\vecx\in\Omega'$ and all $V\geq 10$ we have $|R(V)|\ll_{\ve',\delta}V^{\frac{1}{2}+\delta}$, where the implied constant is independent of $\vecx$. Now, for any $\vecx\in\Omega'$ and all $A\geq10$, we have
\begin{multline}\label{intext}
\bigg|\int_A^{\infty}V^{-2c}\,dR(V)\bigg|=\bigg|\Big[V^{-2c}R(V)\Big]_{V=A}^{V=\infty}+2c\int_{A}^{\infty}V^{-2c-1}R(V)\,dV\bigg|\\
\ll_{\ve',\delta,c_1} A^{-2c+\frac12+\delta}\leq A^{-2c_1+\frac12+\delta},
\end{multline}
uniformly over all $c\in[c_1,c_2]$. Since we can make the right hand side in \eqref{intext} as small as we like, by choosing $A$ large enough, the lemma follows.
\end{proof}

\begin{lem}\label{welldef3}
Let $\frac{1}{4}<c_1<c_2<\frac{1}{2}$. Then, for all $\vecx\in\Omega_{1/2}$ the function $c\mapsto H(c)$ is continuous in $[c_1,\frac12)$. In particular $\mathcal H:\Omega_{1/2}\to C\big([c_1,c_2]\big)$ given by $\vecx\mapsto \big(c\mapsto H(c)\big)$ is a well-defined random function.
\end{lem}

\begin{proof}
Fix $\vecx\in\Omega_{1/2}$. For each $A>0$, the formula \eqref{FA} shows that $c\mapsto\int_0^AV^{-2c}\,dR(V)$ is a continuous function on $[c_1,\frac12)$. Furthermore, by mimicking the proof of Lemma \ref{intest} we see that the function $c\mapsto H(c)$ is the uniform limit of $c\mapsto\int_0^AV^{-2c}\,dR(V)$ as $A\to\infty$. Hence $c\mapsto H(c)$ is indeed continuous in $[c_1,\frac12)$. The second statement now follows from Lemma \ref{welldef1} (cf., e.g.,\ \cite[p.\ 84]{billconv}).
\end{proof}

\begin{remark}
We will also consider $\mathcal H$ as a random function on $\Omega$ (cf.\ Remark \ref{welldef2}).
\end{remark}

\subsection{The random variable $Z_0$}

We now show that the random variable $Z_0$, introduced in \eqref{Z0intro}, is well-defined.

\begin{lem}\label{Z0lemma}
For every $\vecx\in\Omega_{1/2}$ the limit 
\begin{align*}
Z_0:=\lim_{c\to\frac{1}{2}-}\bigg(\int_0^{\infty}V^{-2c}\,dR(V)+\frac{1}{1-2c}\bigg)
\end{align*}
exists. In particular, this limit exists $\textbf P$ almost surely.
\end{lem}

\begin{proof}
For any $\vecx\in\Omega_{1/2}$, $A>0$ and $c\in(\frac14,\frac12)$ we have 
\begin{align}\label{goodrep}
&\int_0^{\infty}V^{-2c}\,dR(V)+\frac{1}{1-2c}\\
&=\int_0^AV^{-2c}\,dN(V)-\int_0^AV^{-2c}\,dV+\int_A^{\infty}V^{-2c}\,dR(V)+\frac{1}{1-2c}\nonumber\\
&=\int_0^AV^{-2c}\,dN(V)+\bigg(-A^{-2c}R(A)+2c\int_A^{\infty}V^{-2c-1}R(V)\,dV\bigg)+\frac{1-A^{1-2c}}{1-2c}.\nonumber
\end{align}
We recall that $\int_0^AV^{-2c}\,dN(V)=2\sum_{T_j\leq A}T_j^{-2c}$ is a finite sum and note that the integral $\int_A^{\infty}V^{-2c-1}R(V)\,dV$ is absolutely convergent. Hence, for any $\vecx$ and $A$ as above, we can let $c\to\frac12$ in the last line of \eqref{goodrep} to obtain 
\begin{align}\label{formulaZ0}
&\lim_{c\to\frac{1}{2}-}\bigg(\int_0^{\infty}V^{-2c}\,dR(V)+\frac{1}{1-2c}\bigg)\\
&=\int_0^AV^{-1}\,dN(V)+\bigg(-A^{-1}R(A)+\int_A^{\infty}V^{-2}R(V)\,dV\bigg)-\Big(\frac d{dt}A^t\Big)_{|t=0}\nonumber\\
&=\int_0^AV^{-1}\,dN(V)+\int_A^{\infty}V^{-1}\,dR(V)-\log A.\nonumber
\end{align}
Since $\textbf P(\Omega_{1/2})=1$ the proof is complete.
\end{proof}

\begin{remark}\label{Z0remark}
Since the restriction of $Z_0$ to $\Omega_{1/2}$ is (by definition) a pointwise limit of measurable functions, we find that $Z_0$ is a random variable on $\Omega_{1/2}$. In fact we will consider $Z_0$ also as a random variable on $\Omega$ (cf.\ Remark \ref{welldef2}). 
\end{remark}

\begin{remark}\label{Z0formula}
We note that the last line of \eqref{formulaZ0} gives a formula for $Z_0$ for any $A>0$. In particular we have
\begin{align*}
Z_0=\int_0^1V^{-1}\,dN(V)+\int_1^{\infty}V^{-1}\,dR(V).
\end{align*}
\end{remark}

\subsection{$H(c)$ and $Z_0$ have stable distributions}

Even though the random variable $H(c)$ has a rather complicated definition, its distribution can be understood in very explicit terms. More precisely it follows from \cite[Thm.\ 1.4.5]{ST} (slightly
modified to allow for the Poisson process to have intensity $\frac12$)
that $H(c)$ has the strictly $\frac{1}{2c}$-stable distribution 
\begin{align}\label{stability}
S_{\frac{1}{2c}}\bigg(2\bigg(\frac{\Gamma(2-\frac{1}{2c})\cos(\frac{\pi}{4c})}{2(1-\frac{1}{2c})}\bigg)^{2c},1,0\bigg).
\end{align}
(Here we use the same parameterization of stable distributions as \cite{ST}.)

\begin{remark}
Recalling from the introduction the relation between $H(c)$ and the random variable $\int_{0}^{\infty}V^{-2c}\,dN(V)$, defined for $c>\frac12$, it is interesting to note that also $\int_{0}^{\infty}V^{-2c}\,dN(V)$ has a strictly $\frac{1}{2c}$-stable distribution given by the expression  \eqref{stability} (cf.\ \cite[Sec.\ 2.5]{epstein1}). 
\end{remark}

\begin{remark}\label{normalremark}
It follows from \eqref{stability} and \cite[Property 1.2.3]{ST} that, for any $c\in(\frac14,\frac12)$, the random variable $\big(2c-\frac12\big)^{\frac12}H(c)$ has the strictly $\frac{1}{2c}$-stable distribution
\begin{align*}
S_{\frac{1}{2c}}\bigg(2\big(2c-\sfrac12\big)^{\frac12}\bigg(\frac{\Gamma(2-\frac{1}{2c})\cos(\frac{\pi}{4c})}{2(1-\frac{1}{2c})}\bigg)^{2c},1,0\bigg). 
\end{align*}
Hence, since
\begin{align*} 
\lim_{c\to\frac14+}\bigg(\frac{1}{2c},2\big(2c-\sfrac12\big)^{\frac12}\bigg(\frac{\Gamma(2-\frac{1}{2c})\cos(\frac{\pi}{4c})}{2(1-\frac{1}{2c})}\bigg)^{2c},1,0\bigg)=\big(2,\sfrac{1}{\sqrt 2},1,0\big) 
\end{align*}
and $S_2\big(\frac{1}{\sqrt2},1,0\big)=S_2\big(\frac{1}{\sqrt2},0,0\big)=N(0,1)$, we conclude, using \cite[Def.\ 1.1.6]{ST} and \cite[Thm.\ 26.3]{billing}, that $\big(2c-\frac12\big)^{\frac12}H(c)$ converges in distribution to $N(0,1)$ as $c\to\frac14+$. Note in particular that this proves Theorem \ref{independentgaussians} in the case $m=1$.
\end{remark}

\begin{figure}[t]
\begin{center}$
\begin{array}{cc}
\includegraphics[width=0.5\textwidth]{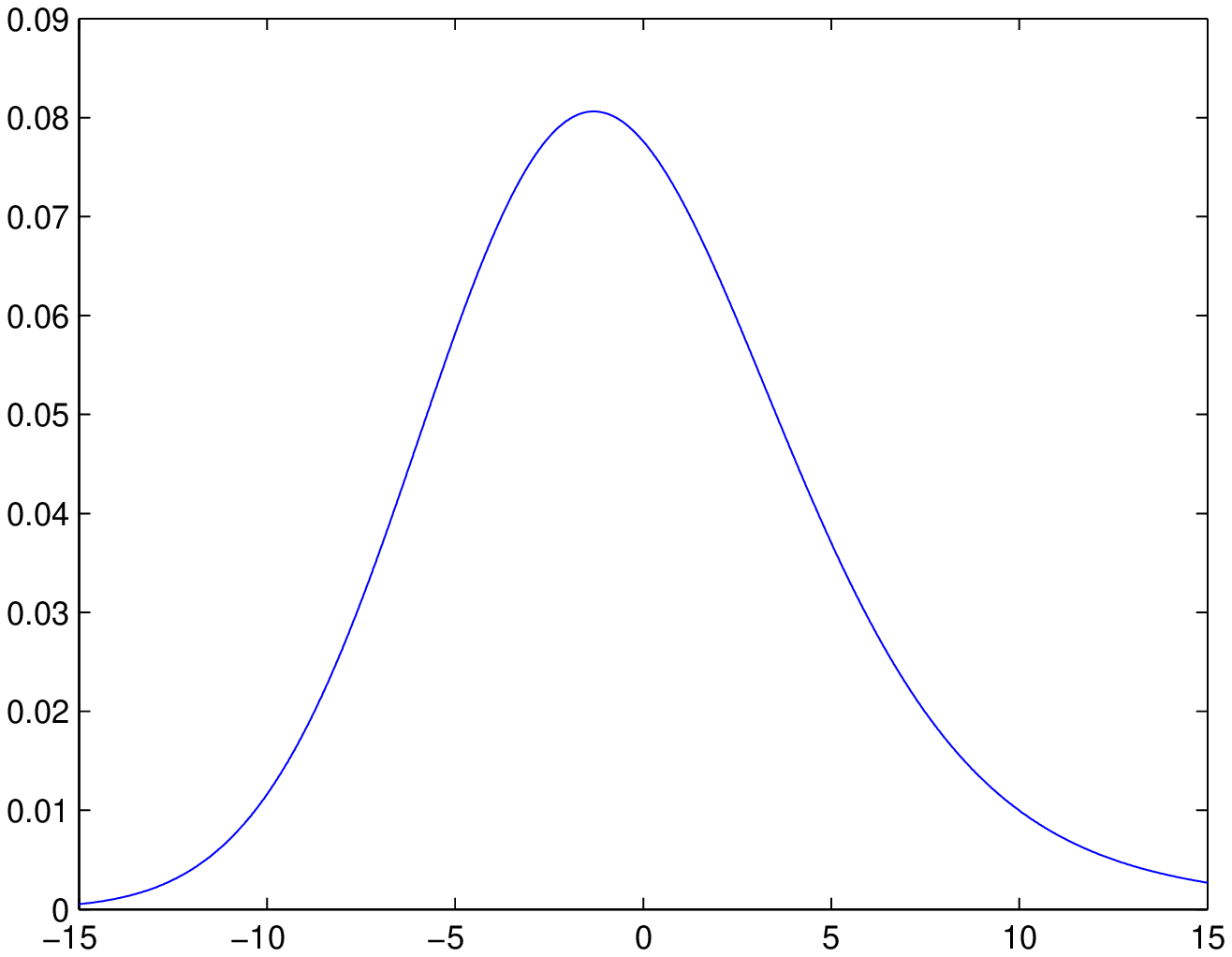}&
\includegraphics[width=0.5\textwidth]{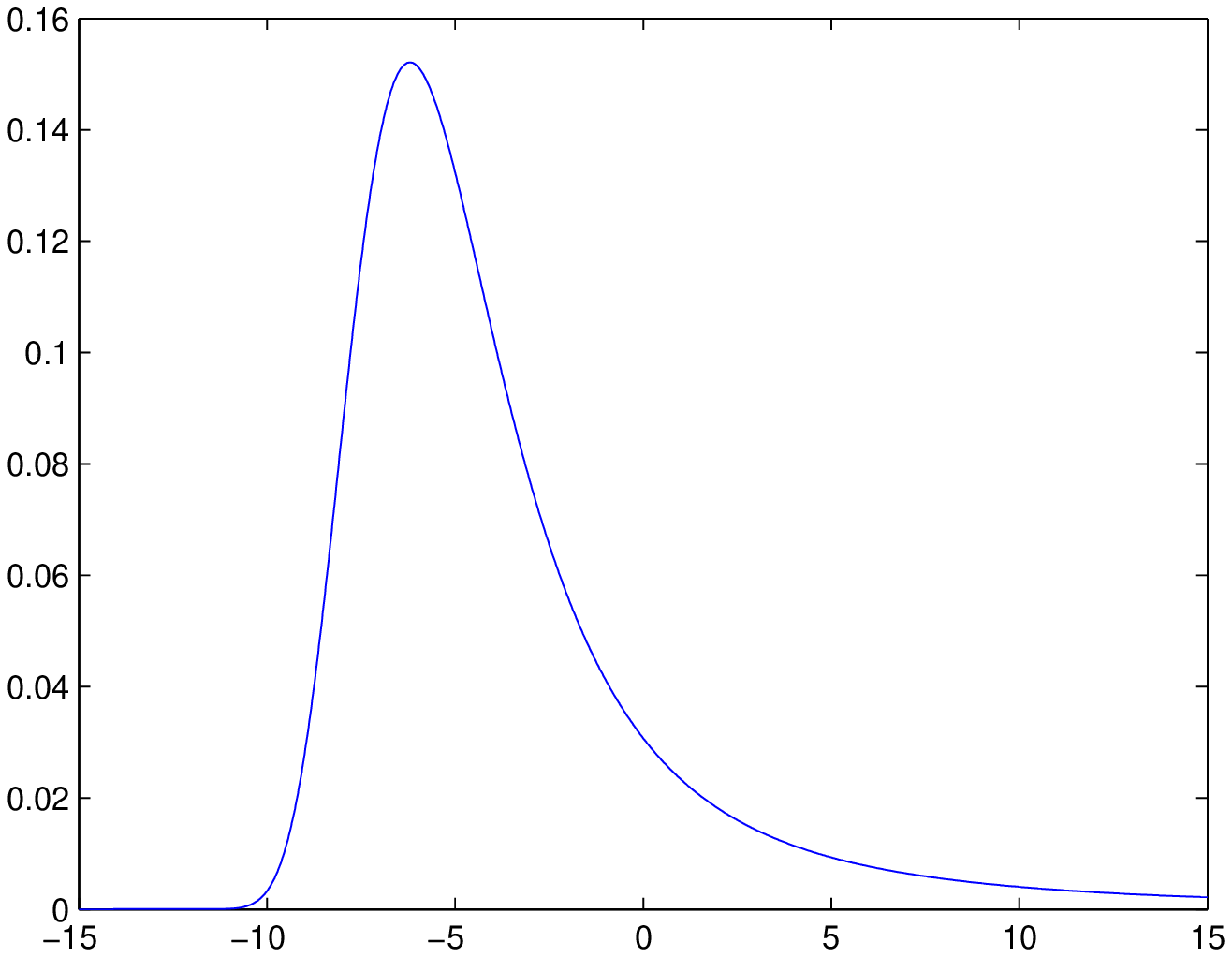}
\end{array}$
\caption{The probability density functions of $H\big(\frac{5}{18}\big)$ (left) and $H\big(\frac{5}{12}\big)$ (right). The figures were generated by the program STABLE, which is available from J. P. Nolan's website http://academic2.american.edu/$\sim$jpnolan/.} \label{Z0fig}
\end{center}
\end{figure}

By an argument similar to the one in Remark \ref{normalremark} we now show that also $Z_0$ has a stable distribution. First we define, for each $c\in(\frac14,\frac12)$, the random variable
\begin{align*}
\widehat H(c):=H(c)+\frac{1}{1-2c}\,,
\end{align*}
so that $\widehat H(c)$ tends in distribution to $Z_0$ as $c\to\frac12-$. It follows from \eqref{stability} and \cite[Property 1.2.2]{ST} that $\widehat H(c)$ has the stable distribution $S_{\alpha(c)}(\sigma(c),\beta(c),\mu(c))$, where
\begin{align}\label{intparameters}
\big(\alpha(c),\sigma(c),\beta(c),\mu(c)\big)=\bigg(\frac{1}{2c},2\bigg(\frac{\Gamma(2-\frac{1}{2c})\cos(\frac{\pi}{4c})}{2(1-\frac{1}{2c})}\bigg)^{2c},1,\frac{1}{1-2c}\bigg). 
\end{align}
Note that, since $\lim_{c\to\frac12-}\alpha(c)=1$ and the characteristic function for a stable distribution (in this parameterization) does not vary continuously with respect to $\alpha$ at $\alpha=1$, we cannot take the limit directly in \eqref{intparameters}. However, using \cite[p.\ 7, Rem.\ 4]{ST}, we find that $\widehat H(c)$ tends in distribution to $S_{\alpha}(\sigma,\beta,\mu)$, where
\begin{align*}
(\alpha,\sigma,\beta,\mu)&=\lim_{c\to\frac12-}\Big(\alpha(c),\sigma(c),\beta(c),\mu(c)+\beta(c)\sigma(c)^{\alpha(c)}\tan\Big(\frac{\pi\alpha(c)}{2}\Big)\Big)\\
&=\big(1,\sfrac{\pi}{2},1,1-\log2-\gamma\big).
\end{align*} 
(Here $\gamma$ is Euler's constant.) Hence we conclude that $Z_0$ has the $1$-stable distribution $S_1(\sfrac{\pi}{2},1,1-\log2-\gamma)$.

\begin{figure}[t]
\begin{center}
\includegraphics[width=0.8\textwidth]{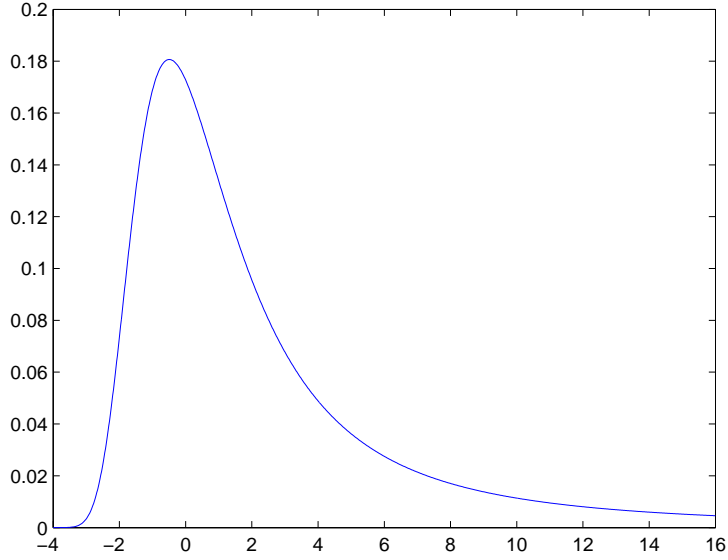}
\caption{The probability density function of $Z_0$. The figure was generated by the program STABLE, which is available from J. P. Nolan's website http://academic2.american.edu/$\sim$jpnolan/.} \label{Z0fig}
\end{center}
\end{figure}

\section{Proof of Theorem \ref{Rnthm}}\label{Rnsec}

Recall that the proof of Lemma \ref{Rest} only uses the monotonicity of $N(V)$ and the variance relation \eqref{variance} (where an upper bound ``$\ll V_2-V_1$'' suffices), and makes no further use of the fact that $R(V)$ is defined in terms of a Poisson process. For this reason, it turns out that the proof of Theorem \ref{Rnthm} can be completed by a direct mimic of the proof of Lemma \ref{Rest}, once we have Lemma \ref{Rnlem1} below. 

\begin{lem}\label{Rnlem1}
For all $A\geq0$, $\Delta>0$ and $n\geq3$ we have 
\begin{align}\label{Rnlem2}
\mathbb E\Big(\big(R_n(A+\Delta)-R_n(A)\big)^2\Big)<5\Delta.
\end{align}
\end{lem} 

Note that it follows from Siegel's mean value formula \cite{siegel} that
\begin{align*}
\mathbb E\big(R_n(A+\Delta)-R_n(A)\big)=0,
\end{align*}
and hence also that the left hand side of \eqref{Rnlem2} equals the variance of $R_n(A+\Delta)-R_n(A)$.

\begin{proof}[Proof of Lemma \ref{Rnlem1}]
Recall that $V_n$ denotes the volume of the unit ball in $\R^n$ and that $V_n=\omega_n/n$, where $\omega_n$ is the $(n-1)$-dimensional volume of the unit sphere $S^{n-1}\subset\R^n$.
To begin with we note that 
\begin{multline*}
\mathbb E\Big(\big(R_n(A+\Delta)-R_n(A)\big)^2\Big)
=\mathbb E\Big(\big(N_n(A+\Delta)-N_n(A)\big)^2\Big)-\Delta^2\\
=\int_{X_n}\sum_{\vecm_1,\vecm_2\in L}I\Big(V_n|\vecm_1|^n,V_n|\vecm_2|^n\in(A,A+\Delta] \Big)\,d\mu_n(L)-\Delta^2.
\end{multline*}
Now recall that for any nonnegative Borel measurable function $\rho$ on $\R^n\times\R^n$ satisfying $\rho(\pm\vecx_1,\pm\vecx_2)=\rho(\vecx_1,\vecx_2)$, Rogers' mean value formula states that (cf.\ \cite[Thm.\ 4]{rogers})
\begin{align}\label{rogers}
&\int_{X_n}\sum_{\vecm_1,\vecm_2\in L\setminus\{\vec0\}}\rho(\vecm_1,\vecm_2)\,d\mu_n(L)\nonumber\\
&=\int_{\R^n}\int_{\R^n}\rho(\vecx_1,\vecx_2)\,d\vecx_1d\vecx_2+\sum_{e_1=1}^{\infty}\sum_{\substack{e_2\in\Z\setminus\{0\}\\ \gcd(e_1,e_2)=1}}\frac{1}{e_1^n}\int_{\R^n}\rho\Big(\vecx,\frac{e_2}{e_1}\vecx\Big)\,d\vecx\\
&=\int_{\R^n}\int_{\R^n}\rho(\vecx_1,\vecx_2)\,d\vecx_1d\vecx_2+\frac{2}{\zeta(n)}\sum_{d_1=1}^{\infty}\sum_{d_2=1}^{\infty}\int_{\R^n}\rho(d_1\vecx,d_2\vecx)\,d\vecx.\nonumber
\end{align}
Applying \eqref{rogers} with the function 
\begin{align*}
\rho(\vecx_1,\vecx_2):=I\Big(|\vecx_1|,|\vecx_2|\in\big(V_n^{-\frac1n}A^{\frac1n},V_n^{-\frac1n}(A+\Delta)^{\frac1n}\big] \Big)
\end{align*}
yields
\begin{align*}
&\mathbb E\Big(\big(R_n(A+\Delta)-R_n(A)\big)^2\Big)=\frac{2}{\zeta(n)}\sum_{d_1=1}^{\infty}\sum_{d_2=1}^{\infty}\int_{\R^n}\rho(d_1\vecx,d_2\vecx)\,d\vecx\\
&=\frac{2}{\zeta(n)}\sum_{d_1=1}^{\infty}\sum_{d_2=1}^{\infty}\omega_n\int_0^{\infty}I\Big(d_1r,d_2r\in\big(V_n^{-\frac1n}A^{\frac1n},V_n^{-\frac1n}(A+\Delta)^{\frac1n}\big] \Big)r^{n-1}\,dr\\
&=\frac{2\omega_n}{\zeta(n)}\sum_{1\leq d_1\leq d_2}\frac{2-I(d_1=d_2)}{d_1^n}\int_0^{\infty}I\Big(u,\frac{d_2}{d_1}u\in\big(V_n^{-\frac1n}A^{\frac1n},V_n^{-\frac1n}(A+\Delta)^{\frac1n}\big] \Big)u^{n-1}\,du\\
&=\frac{2\omega_n}{\zeta(n)}\sum_{1\leq d_1\leq d_2<(1+\Delta/A)^{1/n}d_1}\frac{2-I(d_1=d_2)}{d_1^n}\Big[\frac{u^n}{n}\Big]_{u=V_n^{-\frac1n}A^{\frac1n}}^{u=V_n^{-\frac1n}(d_1/d_2)(A+\Delta)^{\frac1n}}\\
&=\frac{2}{\zeta(n)}\sum_{1\leq d_1\leq d_2<(1+\Delta/A)^{1/n}d_1}\frac{2-I(d_1=d_2)}{d_1^n}\Big(\Big(\frac{d_1}{d_2}\Big)^n(A+\Delta)-A\Big)\\
&=2\Delta+\frac{4}{\zeta(n)}\sum_{1\leq d_1<d_2<(1+\Delta/A)^{1/n}d_1}d_1^{-n}\Big(\Big(\frac{d_1}{d_2}\Big)^n(A+\Delta)-A\Big).
\end{align*}
Note that for $1\leq d_1<d_2$ we have $\big(\frac{d_1}{d_2}\big)^n(A+\Delta)-A<\big(\frac{d_1}{d_2}\big)^n\Delta$. Hence
\begin{align*}
\mathbb E\Big(\big(R_n(A+\Delta)-R_n(A)\big)^2\Big)&<2\Delta+\frac{4}{\zeta(n)}\sum_{d_1=1}^{\infty}\sum_{d_2=d_1+1}^{\infty}d_2^{-n}\Delta\\
&<2\Delta+\frac{4}{\zeta(n)}\sum_{d_1=1}^{\infty}\bigg(\int_{d_1}^{\infty}x^{-n}\,dx\bigg)\Delta\\
&=\bigg(2+\frac{4\zeta(n-1)}{(n-1)\zeta(n)}\bigg)\Delta<5\Delta,
\end{align*}
which is the desired bound.
\end{proof}

\section{Treatment of the Epstein zeta function}\label{Epsteinsection}

When working with the Epstein zeta function in the critical strip it is often convenient to consider the normalized function $F_n(L,s)$ (cf.\ \eqref{Fn}). In particular this function has a simple expansion into incomplete gamma functions (cf.\ \cite[Thm.\ 2]{terras});
\begin{align}\label{Fn2}
 F_n(L,s)=\bigg(-\frac{1}{\sfrac{n}{2}-s}+{\sum_{\vecm\in L}}'G\big(s,\pi |\vecm|^2\big)\bigg)
+\bigg(-\frac{1}{s}+{\sum_{\vecm\in L^*}}'G\big(\sfrac{n}{2}-s,\pi |\vecm|^2\big)\bigg)
\end{align}
holds for $s\in\C\setminus\{0,\frac n2\}$, where 
\begin{align*}
G(s,x):=\int_1^{\infty}t^{s-1}e^{-xt}\,dt, \qquad \Re x>0.
\end{align*}
We define
\begin{align}\label{Hn}
 H_n(L,s):=-\frac{1}{\sfrac{n}{2}-s}+{\sum_{\vecm\in L}}'G\big(s,\pi |\vecm|^2\big),
 \end{align}
and thus the identity \eqref{Fn2} becomes
\begin{align}\label{Fn3}
F_n(L,s)=H_n(L,s)+H_n(L^*,\sfrac{n}{2}-s).
\end{align} 
Hence, to be able to understand the function $F_n(L,s)$ we need first to understand the function $H_n(L,s)$. As a first step, we observe that the integral obtained by replacing the summation over $L$ in \eqref{Hn} by integration over $\R^n$ can be evaluated explicitly:

\begin{lem}\label{int}
 For each $s\in\C$ with $\Re s<\frac n2$ we have
\begin{align*}
 \int_{\R^n}G\big(s,\pi|\vecx|^2\big)\,d\vecx=\frac{1}{\sfrac{n}{2}-s}.
\end{align*}
\end{lem}

\begin{proof}
Changing to spherical coordinates we have (recalling that $\omega_n$ denotes the $(n-1)$-dimensional volume of the unit sphere $S^{n-1}\subset\R^n$)
\begin{align}\label{INTGSXFORMULA}
\int_{\R^n}G\big(s,\pi|\vecx|^2\big)\,d\vecx&=\omega_n\int_0^{\infty}G\big(s,\pi r^2\big)r^{n-1}\,dr\\
&=\frac{\omega_n}{2}\pi^{-\frac{n}{2}}\int_0^{\infty}G(s,x)x^{\frac{n}{2}-1}\,dx\nonumber\\
&=\frac{\omega_n}{2}\pi^{-\frac{n}{2}}\int_0^{\infty}\int_1^{\infty}t^{s-1}e^{-xt}x^{\frac{n}{2}-1}\,dtdx\nonumber\\
&=\frac{\omega_n}{2}\pi^{-\frac{n}{2}}\Gamma(\sfrac{n}{2})\int_1^{\infty}t^{s-\frac{n}{2}-1}\,dt=\frac{1}{\sfrac{n}{2}-s},\nonumber
\end{align}
where we in the last step use the well-known identity $\omega_n=2\pi^{\frac{n}{2}}\Gamma(\sfrac{n}{2})^{-1}$.
\end{proof}

It follows from Siegel's mean value formula \cite{siegel} that the expectation value of the sum over $L$ in \eqref{Hn} equals the integral in Lemma \ref{int}, and hence we have:
\begin{align}\label{EXPC0}
\mathbb E\big(H_n(\cdot,s)\big)=0\qquad\text{for all $s$ with $\Re s<\sfrac n2$.}
\end{align}
In fact, for real $s$ all terms in the sum in \eqref{Hn} are positive, and we will see in the proof of Theorem \ref{curvethm} that for most lattices $L\in X_n$ with $n$ large, and $s\in(\frac n4,\frac n2)$, we have exponential cancellation between the sum and the term $-(\frac n2-s)^{-1}$: For any fixed $c\in(\frac14,\frac12)$ there exists some $\delta>0$ such that 
\begin{align}\label{EXPCANCELLATION}
\text{Prob}_{\mu_n}\Bigl\{L\in X_n\:\Big|\:\bigl|H_n(L,cn)\bigr|<e^{-\delta n}\Bigr\}\to1
\end{align}
as $n\to\infty$. (Cf.\ Remark \ref{cancellationremark} below.) Hence the analysis of $H_n(L,s)$ is quite delicate.

The key to capturing the exponential cancellation in \eqref{Hn} and getting control on the difference $H_n(L,s)$ is our Theorem \ref{Rnthm}, and our starting point is to rewrite \eqref{Hn} in terms of $R_n(V)$. Note that Lemma \ref{int} can be expressed as
\begin{align*}
\int_0^\infty G\Bigl(s,\pi\Bigl(\frac{nV}{\omega_n}\Bigr)^{\frac 2n}\Bigr)\,dV=\frac1{\frac n2-s}
\end{align*}
(indeed, substituting $x=\pi\big(\frac{nV}{\omega_n}\big)^{\frac2n}$ in the integral we get back the second line in \eqref{INTGSXFORMULA} above). Hence, recalling the definitions of $N_n(V)$ and $R_n(V)$ from the introduction, we have 
\begin{align}\label{Hint}
H_n(L,s)=-\frac{1}{\sfrac{n}{2}-s}+\int_0^{\infty}G\Big(s,\pi\Big(\frac{nV}{\omega_n}\Big)^{\frac{2}{n}}\Big)\,dN_n(V)
=\int_0^{\infty}G\Big(s,\pi\Big(\frac{nV}{\omega_n}\Big)^{\frac{2}{n}}\Big)\,dR_n(V),
\end{align}
for all $s$ with $0<s<\frac n2$. The idea is now that the tail of this integral will be small compared with the size of $H_n(L,s)$. The precise meaning of this statement will be clear below.

\begin{lem}\label{GAMMA}
For $0<x\leq s-1$ we have 
\begin{align*}
x^{-s}\Gamma(s)-e^{-x}\leq G(s,x)\leq x^{-s}\Gamma(s).
\end{align*}
\end{lem}

\begin{proof}
From the definition of $G(s,x)$ we get
\begin{align*}
G(s,x)=\int_1^{\infty}t^{s-1}e^{-xt}\,dt=x^{-s}\int_x^{\infty}u^{s-1}e^{-u}\,du=x^{-s}\bigg(\Gamma(s)-\int_0^xu^{s-1}e^{-u}\,du\bigg).
\end{align*}
Here, since the function $u\mapsto u^{s-1}e^{-u}$ is increasing for $u\in(0,s-1)$, we have $0\leq\int_0^xu^{s-1}e^{-u}\,du\leq x^se^{-x}$ for $0<x\leq s-1$ and the lemma follows.
\end{proof}

Applying Stirling's formula we get
\begin{align}\label{omegan}
\omega_n=\frac{2\pi^{\frac{n}{2}}}{\Gamma(\sfrac{n}{2})}\sim\Big(\frac{2\pi e}{n}\Big)^{\frac n2}\Big(\frac{n}{\pi}\Big)^{\frac12}\qquad\text{ as $n\to\infty$}.
\end{align}
As a consequence we note that $\pi(\frac{n}{\omega_n})^{2/n}\sim\frac{n}{2e}$ as $n\to\infty$ and hence, for fixed $A>0$ and all large enough $n$, we have $\pi(\frac{nA}{\omega_n})^{2/n}<\frac{n}{4}-1$. Thus, for all $c\in[\frac14,\frac12)$ and $A$ and $n$ as above, Lemma \ref{GAMMA} applies to give
\begin{multline}\label{firstA}
\int_0^AG\Big(cn,\pi\Big(\frac{nV}{\omega_n}\Big)^{\frac{2}{n}}\Big)\,dR_n(V)=\Gamma(cn)\pi^{-cn}\Big(\frac{n}{\omega_n}\Big)^{-2c}\int_0^AV^{-2c}\,dR_n(V)\\
+O(1)\int_0^A\exp\Big\{-\pi\Big(\frac{nV}{\omega_n}\Big)^{\frac2n}\Big\}\big(dN_n(V)+dV\big),
\end{multline}
with an absolute implied constant. We choose not to consider this identity for $c=\frac12$ since in that case both the integrals in the first row of \eqref{firstA} are divergent. For notational convenience we set
\begin{align*}
K_{c,n}:=\Gamma(cn)\pi^{-cn}\Big(\frac{n}{\omega_n}\Big)^{-2c}.
\end{align*}

\begin{prop}\label{HAprop}
Let $A>0$ be fixed. Then, for all $k<\frac{1}{2e}$, we have
\begin{multline*}
\text{Prob}_{\mu_n}\bigg\{L\in X_n\,\Big|\,\bigg|K_{c,n}^{-1}\int_0^AG\Big(cn,\pi\Big(\frac{nV}{\omega_n}\Big)^{\frac{2}{n}}\Big)\,dR_n(V)-\int_0^AV^{-2c}\,dR_n(V)\bigg|\\
<K_{c,n}^{-1}e^{-kn}, \,\forall c\in[\sfrac14,\sfrac12)\bigg\}\to1
\end{multline*}
as $n\to\infty$.
\end{prop}

\begin{proof}
We consider the integral with respect to $dN_n(V)$ and the integral with respect to $dV$ separately in the error term in \eqref{firstA}. Changing variables $V=\frac{\omega_n}{n}(\frac{x}{\pi})^{n/2}$ yields
\begin{align}\label{secondA}
\int_0^A\exp\Big\{-\pi\Big(\frac{nV}{\omega_n}\Big)^{\frac2n}\Big\}\,dV=\frac{\omega_n}{2}\pi^{-\frac n2}\int_0^{\pi(nA/\omega_n)^{2/n}}e^{-x}x^{\frac n2-1}\,dx.
\end{align}
Recalling that we have $\pi(nA/\omega_n)^{2/n}<k'n$ for any fixed $k'>\frac{1}{2e}$ and all sufficiently large $n$, as well as the fact that $x\mapsto e^{-x}x^{\frac n2-1}$ is increasing for all $0<x<\frac n2-1$, we find that, for $\frac{1}{2e}<k'<\frac12$ and large enough $n$, \eqref{secondA} is
\begin{align*}
O\bigg(\omega_n\Big(\frac{k'n}{\pi}\Big)^{\frac n2}e^{-k'n}\bigg)=O\bigg(n^{\frac12}\text{exp}\Big(\big(\sfrac{1}{2}\log(2ek')-k'\big)n\Big)\bigg).
\end{align*}
By taking $k'$ sufficiently close to $\frac1{2e}$ it follows that for any fixed $k<\frac{1}{2e}$ there exists $n_0\in\Z_{\geq1}$ (which also depends on $A$) such that 
\begin{align*}
\int_0^A\exp\Big\{-\pi\Big(\frac{nV}{\omega_n}\Big)^{\frac2n}\Big\}\,dV<e^{-kn}
\end{align*}
for all $n\geq n_0$.

Next, let $\ve>0$ be given. By possibly increasing $n_0$ it follows from \cite[Thm.\ 3]{rogers3} (cf.\ also \cite[Thm.\ 1]{poisson}) that there exists $M\in \Z_{\geq1}$ such that for $n\geq n_0$ we have both $N_n(A)<M$ and $N_n(M^{-1})=0$ with probability $>1-\ve$. Since also $(M^{-1})^{2/n}\to1$ as $n\to\infty$, we conclude that for any fixed constant $k<\frac{1}{2e}$ and all $n\geq n_0$ (with a possibly even larger $n_0$ depending on $k$) we have 
\begin{align*}
\int_0^A\exp\Big\{-\pi\Big(\frac{nV}{\omega_n}\Big)^{\frac2n}\Big\}\,dN_n(V)<M\exp\Big\{-\pi\Big(\frac{nM^{-1}}{\omega_n}\Big)^{\frac2n}\Big\}<e^{-kn},
\end{align*}
with probability $>1-\ve$. Hence for our fixed $A>0$ and $k<\frac1{2e}$ and all $n\geq n_0$, the absolute error in \eqref{firstA} is $<Ce^{-kn}$, where $C$ is an absolute constant, with probability $>1-\ve$. Thus for any $k'<k$ the absolute error is $<e^{-k'n}$ for all sufficiently large $n$ with probability $>1-\ve$, and the proposition follows. 
\end{proof}

\begin{remark}\label{Kremark}
We stress that with an appropriate choice of $k$, the upper bound $K_{c,n}^{-1}e^{-kn}$ in Proposition \ref{HAprop} tends to zero as $n\to\infty$, uniformly with respect to $c\in[\frac14,\frac12]$. Indeed, note that 
\begin{align*}
K_{c,n}^{-1}e^{-kn}&=O\bigg(n^{\frac12}\Big(\frac{\pi e}{cn}\Big)^{cn}n^{2c}\Big(\frac{n}{2\pi e}\Big)^{cn}n^{-c}e^{-kn}\bigg)\\
&=O\bigg(n^{\frac12+c}\,\text{exp}\Big(-\big(k+c\log(2c)\big)n\Big)\bigg).
\end{align*} 
Here $k+c\log(2c)\geq k-\frac14\log2=k-0.1732...$ for all $c\in[\frac14,\frac12]$, and the last difference is positive when $k$ is sufficiently close to $\frac{1}{2e}=0.1839...$. 
\end{remark}

Next we estimate the tail of the integral giving $H_n(L,s)$, normalized in the same way as the integral in Proposition \ref{HAprop}. The proof is similar to the proof of Lemma \ref{intest}. We first recall two bounds on $G(s,x)$ which will be used several times in this paper.

\begin{lem}\label{sstbound}
The following bound holds uniformly for all $x>0$, $s\geq1$,
\begin{align*}
G(s,x)\ll s^{-\frac12}\Big(\frac{ex}{s}\Big)^{-s}.
\end{align*}
In the case $x\geq s\geq1$ we also have the stronger bound
\begin{align*}
G(s,x)\ll s^{-\frac12}e^{-x}.
\end{align*}
\end{lem}

\begin{proof}
Cf.\ \cite[Cor.\ 2]{sst}.
\end{proof}

\begin{lem}\label{Htail}
Let $c_1\in(\frac{1}{4},\frac12)$. Then, for all $\ve>0$ there exist constants $A_0>0$ and $n_0\in\Z_{\geq3}$ such that for all $A\geq A_0$  and $n\geq n_0$ we have
\begin{align*}
\text{Prob}_{\mu_n}\,\bigg\{L\in X_n\,\Big|\,\sup_{c\in[c_1,\frac12]}\bigg|\,K_{c,n}^{-1}\int_A^{\infty}G\Big(cn,\pi\Big(\frac{nV}{\omega_n}\Big)^{\frac{2}{n}}\Big)\,dR_n(V)\bigg|\leq\ve\bigg\}\geq1-\ve.
\end{align*}
\end{lem}

\begin{proof}
Let $\ve>0$ and $\delta\in(0,2c_1-\frac12)$ be given. It follows from Theorem \ref{Rnthm} that for each $n\geq3$ there exists a set $X_n'\subset X_n$ with $\mu_n (X_n')\geq 1-\ve$ such that for all $L\in X_n'$ and all $V\geq 10$ we have $|R_n(V)|\ll_{\ve,\delta}V^{\frac{1}{2}+\delta}$, where the implied constant is independent of $n$ and $L$. Now, integrating by parts and using \mbox{$\frac {\partial}{\partial x}G(s,x)=-G(s+1,x)$,} we have
\begin{multline}\label{intbyparts}
\int_A^{\infty}G\Big(cn,\pi\Big(\frac{nV}{\omega_n}\Big)^{\frac{2}{n}}\Big)\,dR_n(V)=\Big[G\Big(cn,\pi\Big(\frac{nV}{\omega_n}\Big)^{\frac{2}{n}}\Big)R_n(V)\Big]_{V=A}^{V=\infty}\\
+\frac{2\pi}{n}\Big(\frac{n}{\omega_n}\Big)^{\frac2n}\int_{A}^{\infty}G\Big(cn+1,\pi\Big(\frac{nV}{\omega_n}\Big)^{\frac{2}{n}}\Big)V^{\frac2n-1}R_n(V)\,dV.
\end{multline}
Hence, using Lemma \ref{sstbound} we get, for any $L\in X_n'$ (with $n$ sufficiently large) and all $A\geq10$,
\begin{multline*}
\bigg|K_{c,n}^{-1}\int_A^{\infty}G\Big(cn,\pi\Big(\frac{nV}{\omega_n}\Big)^{\frac{2}{n}}\Big)\,dR_n(V)\bigg|\\
\ll_{\ve,\delta} A^{-2c+\frac12+\delta}
+\Big(\frac{cn+1}{cn}\Big)^{cn}\int_{A}^{\infty}V^{-2c-\frac12+\delta}\,dV
\ll_{\ve,\delta,c_1} A^{-2c_1+\frac12+\delta},
\end{multline*}
uniformly over all $c\in[c_1,\frac12]$. Thus we can make the left hand side above as small as we like, by choosing $A$ large enough.
\end{proof}

Given $\ve>0$ and $c_1\in(\frac14,\frac12)$, it follows from \eqref{Hint}, Proposition \ref{HAprop} and Lemma \ref{Htail} that there exists $A_0>0$ such that for all $A\geq A_0$ there exists $n_0\in\Z_{\geq3}$ such that for all $n\geq n_0$ we have
\begin{align}\label{important1}
\text{Prob}_{\mu_n}\bigg\{L\in X_n\:\Big|\:\sup_{c\in[c_1,\frac12)}\bigg|K_{c,n}^{-1}H_n(L,cn)-\int_0^AV^{-2c}\,dR_n(V)\bigg|\leq\ve\bigg\}\geq1-\ve.
\end{align}

Since our goal is  to understand the function $F_n(L,cn)$ for $c\in[c_1,\frac12)$ it remains to study $J_n(L,s):=H_n(L,\frac n2-s)$ for $s=cn$ with $c\in[c_1,\frac12)$ (recall \eqref{Fn3}).

\begin{prop}\label{JAprop}
Given any $c_1\in(\frac{1}{4},\frac12)$ there exists a constant $k>0$ such that 
\begin{align*}
\text{Prob}_{\mu_n}\Big\{L\in X_n\,\Big|\,\,K_{c,n}^{-1}\big|J_n(L,cn)\big|<e^{-kn},\,\forall c\in[c_1,\sfrac12]\Big\}\to1
\end{align*}
as $n\to\infty$.
\end{prop}

\begin{proof}
It follows from \eqref{Hint} and integration by parts, together with the estimates in Lemma \ref{sstbound} and the bound $G(s,x)\ll x^{-1}e^{-x}$ for $0\leq s\leq 1$, that
\begin{align*}
J_n(L,cn)&=\int_0^{\infty}G\Big(\frac n2-cn,\pi\Big(\frac{nV}{\omega_n}\Big)^{\frac{2}{n}}\Big)\,dR_n(V)\\
&=\frac{2\pi}{n}\Big(\frac{n}{\omega_n}\Big)^{\frac2n}\int_{0}^{\infty}G\Big(\frac n2-cn+1,\pi\Big(\frac{nV}{\omega_n}\Big)^{\frac{2}{n}}\Big)V^{\frac2n-1}R_n(V)\,dV\nonumber
\end{align*}
(cf.\ \eqref{intbyparts}). Furthermore, changing variables $V=\frac{\omega_n}{n}(\frac{x}{\pi})^{n/2}$ we obtain
\begin{align}\label{Jint}
J_n(L,cn)=\int_{0}^{\infty}G\Big(\frac n2-cn+1,x\Big)R_n\Big(\frac{\omega_n}{n}\Big(\frac{x}{\pi}\Big)^{\frac n2}\Big)\,dx.
\end{align}
Given $\ve,\delta\in(0,\frac12)$, it follows from \cite{poisson} and Theorem \ref{Rnthm} that there exist $n_0\in\Z_{\geq3}$, $M\in\Z_{\geq1}$ and sets $X_n''\subset X_n$ with $\mu_n(X_n'')>1-\ve$ such that for all $n\geq n_0$ and $L\in X_n''$ we have $N_n(V)=0$ for all $V\in[0,M^{-1}]$, $N_n(10)<M$, and $|R_n(V)|\ll_{\ve,\delta} V^{\frac12+\delta}$ for all $V\geq10$. It follows that, for all $n\geq n_0$ and $L\in X_n''$, we have 
\begin{align*}
|R_n(V)|\ll_{\ve,\delta}\min\Bigl(V,V^{\frac12+\delta}\Bigr),\qquad\forall V\geq0.
\end{align*}
We now estimate $J_n(L,cn)$ for all $c\in[\frac14,\frac12]$ and $L\in X_n''$ ($n\geq n_0$) by splitting the integral in \eqref{Jint} into two parts. More precisely, for $n\geq n_0$ and $L\in X_n''$, we have
\begin{multline}\label{Jint2}
|J_n(L,cn)|\ll\int_{0}^{W_n}G\Big(\frac n2-cn+1,x\Big)\frac{\omega_n}{n}\Big(\frac{x}{\pi}\Big)^{\frac n2}\,dx\\
+\int_{W_n}^{\infty}G\Big(\frac n2-cn+1,x\Big)\bigg(\frac{\omega_n}{n}\Big(\frac{x}{\pi}\Big)^{\frac n2}\bigg)^{\frac12+\delta}\,dx,
\end{multline}
where $W_n=\pi\big(\frac{n}{\omega_n}\big)^{\frac{2}{n}}$. In \eqref{Jint2} and in all other "$\ll$" bounds below, the implied constant may depend on $\ve,\delta$, but is independent of $n,L,c$ (subject to $c\in[\frac14,\frac12]$). Recall here that $W_n\sim\frac{n}{2e}$ as $n\to\infty$. We call the integrals in \eqref{Jint2} $I_1$ and $I_2$ respectively.

To begin with we set $T_n=\min(\frac n2-cn+1,W_n)$ and use Lemma \ref{sstbound} to get
\begin{multline}\label{JI1}
I_1\ll\frac{\omega_n}{n}\pi^{-\frac n2}\big(\sfrac{n}{2}-cn+1\big)^{(\frac 12-c)n+\frac12}e^{(c-\frac12)n}\int_0^{T_n}x^{cn-1}\,dx\\
+\frac{\omega_n}{n}\pi^{-\frac n2}\big(\sfrac{n}{2}-cn+1\big)^{-\frac12}\int_{T_n}^{W_n}x^{\frac n2}e^{-x}\,dx.
\end{multline} 
When $T_n=W_n$ the second integral in \eqref{JI1} vanishes and we have
\begin{align}\label{finalI11}
I_1&\ll \frac{\omega_n}{n^2}\pi^{-\frac n2}\big(\sfrac{n}{2}-cn+1\big)^{(\frac 12-c)n+\frac12}e^{(c-\frac12)n}W_n^{cn}\\
&\ll n^{-1}\Big(\frac{\omega_n}{n}\Big)^{1-2c}(\pi e)^{(c-\frac12)n}\big(\sfrac{n}{2}-cn+1\big)^{(\frac 12-c)n+\frac12}.\nonumber
\end{align}
On the other hand, when $T_n=\frac n2-cn+1$ we have
\begin{align}\label{intermedI1}
I_1&\ll\frac{\omega_n}{n^2}\pi^{-\frac n2}e^{(c-\frac12)n}\big(\sfrac{n}{2}-cn+1\big)^{\frac n2+\frac12}+\frac{\omega_n}{n}\pi^{-\frac n2}\big(\sfrac{n}{2}-cn+1\big)^{-\frac12}\int_{0}^{W_n}x^{\frac n2}e^{-x}\,dx.
\end{align}
One checks that $x\mapsto x^{\frac n2}e^{-x}$ is increasing for all $0<x<\frac n2$. In addition $W_n<\frac n2$ for all large enough $n$. Hence after possibly increasing $n_0$, we have that for all $n\geq n_0$, \eqref{intermedI1} is
\begin{align}\label{finalI12}
&\ll\frac{\omega_n}{n^2}\pi^{-\frac n2}e^{(c-\frac12)n}\big(\sfrac{n}{2}-cn+1\big)^{\frac n2+\frac12}+\frac{\omega_n}{n}\pi^{-\frac n2}\big(\sfrac{n}{2}-cn+1\big)^{-\frac12}W_n^{\frac n2+1}e^{-W_n}\\
&\ll\frac{\omega_n}{n^2}\pi^{-\frac n2}e^{(c-\frac12)n}\big(\sfrac{n}{2}-cn+1\big)^{\frac n2+\frac12}+\omega_n^{-\frac 2n}\big(\sfrac{n}{2}-cn+1\big)^{-\frac12}e^{-W_n}.\nonumber
\end{align}

Next we estimate the integral $I_2$. We set $S_n=\max(\frac n2-cn+1,W_n)$ and use Lemma \ref{sstbound} to get
\begin{multline}\label{JI2}
I_2\ll\Big(\frac{\omega_n}{n}\Big)^{\frac12+\delta}\pi^{-(\frac 14+\frac{\delta}{2})n}\big(\sfrac{n}{2}-cn+1\big)^{(\frac 12-c)n+\frac12}e^{(c-\frac12)n}\int_{W_n}^{S_n}x^{(c-\frac14+\frac{\delta}{2})n-1}\,dx\\
+\Big(\frac{\omega_n}{n}\Big)^{\frac12+\delta}\pi^{-(\frac 14+\frac{\delta}{2})n}\big(\sfrac{n}{2}-cn+1\big)^{-\frac12}\int_{S_n}^{\infty}x^{(\frac 14+\frac{\delta}{2})n}e^{-x}\,dx.
\end{multline} 
When $S_n=W_n$ the first integral in \eqref{JI2} vanishes and we obtain, estimating the function $g(x)=x^{(\frac 14+\frac{\delta}{2})n+3}e^{-x}$ with its maximum, 
\begin{align}\label{finalI21}
I_2&\ll\Big(\frac{\omega_n}{n}\Big)^{\frac12+\delta}\pi^{-(\frac 14+\frac{\delta}{2})n}\big(\sfrac{n}{2}-cn+1\big)^{-\frac12}\int_{W_n}^{\infty}x^{(\frac 14+\frac{\delta}{2})n+3}e^{-x}\,\frac{dx}{x^3}\\
&\ll\Big(\frac{\omega_n}{n}\Big)^{\frac12+\delta}W_n^{-2}\pi^{-(\frac 14+\frac{\delta}{2})n}\big(\sfrac{n}{2}-cn+1\big)^{-\frac12}\Big(\big(\sfrac{1}{4}+\sfrac{\delta}{2}\big)n+3\Big)^{(\frac 14+\frac{\delta}{2})n+3}e^{-(\frac14+\frac{\delta}{2})n}\nonumber\\
&\ll\Big(\frac{\omega_n}{n}\Big)^{\frac12+\delta+\frac4n}\pi^{-(\frac 14+\frac{\delta}{2})n}\big(\sfrac{n}{2}-cn+1\big)^{-\frac12}\Big(\big(\sfrac{1}{4}+\sfrac{\delta}{2}\big)n+3\Big)^{(\frac 14+\frac{\delta}{2})n+3}e^{-(\frac14+\frac{\delta}{2})n}.\nonumber
\end{align}
In the remaining case, that is $S_n=\frac n2-cn+1$, we have
\begin{align}\label{finalI22}
I_2&\ll\Big(\frac{\omega_n}{n}\Big)^{\frac12+\delta}\pi^{-(\frac 14+\frac{\delta}{2})n}\big(\sfrac{n}{2}-cn+1\big)^{(\frac 12-c)n+\frac12}e^{(c-\frac12)n}\int_{0}^{\frac n2-cn+1}x^{(c-\frac14+\frac{\delta}{2})n-1}\,dx\\
&+\Big(\frac{\omega_n}{n}\Big)^{\frac12+\delta}\pi^{-(\frac 14+\frac{\delta}{2})n}\big(\sfrac{n}{2}-cn+1\big)^{-\frac52}\Big(\big(\sfrac{1}{4}+\sfrac{\delta}{2}\big)n+3\Big)^{(\frac 14+\frac{\delta}{2})n+3}e^{-(\frac14+\frac{\delta}{2})n}\nonumber\\
&\ll n^{-1}\Big(\frac{\omega_n}{n}\Big)^{\frac12+\delta}\pi^{-(\frac 14+\frac{\delta}{2})n}e^{(c-\frac12)n}\big(\sfrac{n}{2}-cn+1\big)^{(\frac 14+\frac{\delta}{2})n+\frac12}\nonumber\\
&+\Big(\frac{\omega_n}{n}\Big)^{\frac12+\delta}\pi^{-(\frac 14+\frac{\delta}{2})n}\big(\sfrac{n}{2}-cn+1\big)^{-\frac52}\Big(\big(\sfrac{1}{4}+\sfrac{\delta}{2}\big)n+3\Big)^{(\frac 14+\frac{\delta}{2})n+3}e^{-(\frac14+\frac{\delta}{2})n}.\nonumber
\end{align}
(Recall that the implied constant is allowed to depend on $\delta$.) 

Collecting the results in \eqref{finalI11}, \eqref{finalI12}, \eqref{finalI21} and \eqref{finalI22} we get, for  all $n\geq n_0$, $L\in X_n''$ and $c\in[\frac14,\frac12]$,
\begin{align}\label{klar1}
|J_n(L,cn)|&\ll n^{-1}\Big(\frac{\omega_n}{n}\Big)^{1-2c}(\pi e)^{(c-\frac12)n}\big(\sfrac{n}{2}-cn+1\big)^{(\frac 12-c)n+\frac12}\\
&+n^{-1}\Big(\frac{\omega_n}{n}\Big)^{\frac12+\delta}\pi^{-(\frac 14+\frac{\delta}{2})n}e^{(c-\frac12)n}\big(\sfrac{n}{2}-cn+1\big)^{(\frac 14+\frac{\delta}{2})n+\frac12}\nonumber\\
&+\Big(\frac{\omega_n}{n}\Big)^{\frac12+\delta}\pi^{-(\frac 14+\frac{\delta}{2})n}\big(\sfrac{n}{2}-cn+1\big)^{-\frac52}\Big(\big(\sfrac{1}{4}+\sfrac{\delta}{2}\big)n+3\Big)^{(\frac 14+\frac{\delta}{2})n+3}e^{-(\frac14+\frac{\delta}{2})n}\nonumber
\end{align}
when $W_n\leq \frac n2-cn+1$, and
\begin{align}\label{klar2}
|J_n(L,cn)|&\ll\frac{\omega_n}{n^2}\pi^{-\frac n2}e^{(c-\frac12)n}\big(\sfrac{n}{2}-cn+1\big)^{\frac n2+\frac12}+\omega_n^{-\frac 2n}\big(\sfrac{n}{2}-cn+1\big)^{-\frac12}e^{-W_n}\\
&+\Big(\frac{\omega_n}{n}\Big)^{\frac12+\delta+\frac4n}\pi^{-(\frac 14+\frac{\delta}{2})n}\big(\sfrac{n}{2}-cn+1\big)^{-\frac12}\Big(\big(\sfrac{1}{4}+\sfrac{\delta}{2}\big)n+3\Big)^{(\frac 14+\frac{\delta}{2})n+3}e^{-(\frac14+\frac{\delta}{2})n}\nonumber
\end{align}
when $\frac n2-cn+1\leq W_n$.

It now remains to prove that all terms in \eqref{klar1} and \eqref{klar2} are as small as the proposition claims. We will prove that there exists a constant $k>0$ such that if $\delta$ has been fixed to be sufficiently small (as depends only on $c_1$), then for all sufficiently large $n$ we have
$K_{c,n}^{-1}|J_n(L,cn)|\ll n^{4}e^{-2kn}$ for all $c\in[c_1,\frac12]$ and $L\in X_n''$. Hence, a fortiori,
$K_{c,n}^{-1}|J_n(L,cn)|<e^{-kn}$ for $n$ large enough, and this completes the proof.

We first consider \eqref{klar1}. Using Stirling's formula  and \eqref{omegan} we get
\begin{multline*}
K_{c,n}^{-1}n^{-1}\Big(\frac{\omega_n}{n}\Big)^{1-2c}(\pi e)^{(c-\frac12)n}\big(\sfrac{n}{2}-cn+1\big)^{(\frac 12-c)n+\frac12}\\
\ll n^{2c-\frac12}\big(\sfrac{1}{2}-c+\sfrac{1}{n}\big)^{\frac12}\text{exp}\big(-f_n(c)\cdot n\big),
\end{multline*}
where
\begin{align*}
f_n(c)=c\log c+\big(2c-\sfrac12\big)\log2+\big(c-\sfrac12\big)\log\big(\sfrac12-c+\sfrac1n\big).
\end{align*}
Using $W_n\sim\frac n{2e}$ and $\frac12-\frac1{2e}=0.316...$ we find that for $n$ sufficiently large the assumption $W_n\leq \frac n2-cn+1$ implies $c\leq0.32$. Moreover, for all $c\in[\frac14,0.32]$ we have
\begin{multline}\label{fncomputation}
f_n'(c)=\log c+2+2\log 2+\log\big(\sfrac12-c+\sfrac1n\big)-\big(\sfrac n2-cn+1\big)^{-1}\\
\geq 2+\log\big(\sfrac12-0.32\big)-\big(\sfrac n2-0.32n\big)^{-1},
\end{multline}
which is positive for $n$ sufficiently large. Hence for $n$ sufficiently large and for all $c\in[c_1,\frac12]$ satisfying $W_n\leq\frac n2-cn+1$, we have (writing $f_\infty(c):=\lim_{n\to\infty} f_n(c)$ and noticing that the computation in \eqref{fncomputation} also proves $f'_\infty>0$ for $c\in[\frac14,0.32]$): 
\begin{align*}
f_n(c)\geq f_n(c_1)>\sfrac12 f_\infty(c_1)>0.
\end{align*}
Hence the first term in \eqref{klar1} is small enough. Continuing, we find that
\begin{multline*}
K_{c,n}^{-1}n^{-1}\Big(\frac{\omega_n}{n}\Big)^{\frac12+\delta}\pi^{-(\frac 14+\frac{\delta}{2})n}e^{(c-\frac12)n}\big(\sfrac{n}{2}-cn+1\big)^{(\frac 14+\frac{\delta}{2})n+\frac12}\\
\ll n^{c-\frac{\delta}{2}-\frac14}\big(\sfrac{1}{2}-c+\sfrac{1}{n}\big)^{\frac12}\text{exp}\big(-g_n(c)\cdot n\big),
\end{multline*}
where
\begin{align*}
g_n(c)=c\log c+\big(c-\sfrac14-\sfrac{\delta}{2}\big)\log2+\big(\sfrac14-\sfrac{\delta}{2}-c\big)-\big(\sfrac14+\sfrac{\delta}{2}\big)\log\big(\sfrac12-c+\sfrac1n\big).
\end{align*}
Here $g'_n(c)>\log(2c)+(2-4c+\frac4n)^{-1}>0$ for all $n\geq10$ and $c\in[\frac14,\frac12]$; hence $g_n(c)\geq g_n(c_1)$ for all $c\in[c_1,\frac12]$. Thus, since for all sufficiently large $n$ and small $\delta$ we have that $g_n(c_1)$ is larger than a positive constant which only depends on $c_1$, the second term in \eqref{klar1} is small enough. Next we note that
\begin{multline}\label{decay}
K_{c,n}^{-1}\Big(\frac{\omega_n}{n}\Big)^{\frac12+\delta}\pi^{-(\frac 14+\frac{\delta}{2})n}\big(\sfrac{n}{2}-cn+1\big)^{-\frac52}\Big(\big(\sfrac{1}{4}+\sfrac{\delta}{2}\big)n+3\Big)^{(\frac 14+\frac{\delta}{2})n+3}e^{-(\frac14+\frac{\delta}{2})n}\\
\ll n^{c-\frac{\delta}{2}+\frac34}\big(\sfrac{1}{2}-c+\sfrac{1}{n}\big)^{-\frac52}\text{exp}\big(-h_n(c)\cdot n\big),
\end{multline}
where
\begin{align*}
h_n(c)=c\log c+\big(c-\sfrac14-\sfrac{\delta}{2}\big)\log2-\big(\sfrac14+\sfrac{\delta}{2}\big)\log\big(\sfrac14+\sfrac{\delta}{2}+\sfrac3n\big).
\end{align*}
Now $h'_n(c)\geq 1-\log2>0$ for all $c\geq\frac14$, independently of $n$ and $\delta$, and thus $h_n(c)\geq h_n(c_1)$ for all $c\in[c_1,\frac12]$; also for all sufficiently large $n$ and small $\delta$ we have that $h_n(c_1)$ is larger than a positive constant which only depends on $c_1$. Thus the third term in \eqref{klar1} is small enough.

We now give a similar treatment of the terms in \eqref{klar2}. First we observe that
\begin{align*}
K_{c,n}^{-1}\frac{\omega_n}{n^2}\pi^{-\frac n2}e^{(c-\frac12)n}\big(\sfrac{n}{2}-cn+1\big)^{\frac n2+\frac12}
\ll n^{c-\frac12}\big(\sfrac{1}{2}-c+\sfrac{1}{n}\big)^{\frac12}\text{exp}\big(-j_n(c)\cdot n\big),
\end{align*}
where
\begin{align*}
j_n(c)=c\log c+\big(c-\sfrac12\big)\log2-c-\sfrac12\log\big(\sfrac12-c+\sfrac1n\big).
\end{align*}
Note that $j'_n(c)=\log(2c)+(1-2c+\frac2n)^{-1}>0$ for all $c\in[\frac14,\frac12]$ and all $n\geq3$. Furthermore, using $\frac12-\frac1{2e}=0.316...$, it follows that for $n$ sufficiently large the assumption $W_n\geq \frac n2-cn+1$ implies $c\geq0.3$. Hence $j_n(c)\geq j_n(0.3)$, and for all $n\geq 1000$ we have $j_n(0.3)\geq j_{1000}(0.3)=0.00240...>0$. Hence the first term in \eqref{klar2} is as small as desired. Next we note that, for all sufficiently large $n$ such that $W_n>(\frac1{2e}-\delta)n$, we have
\begin{align*}
K_{c,n}^{-1}\omega_n^{-\frac2n}\big(\sfrac n2-cn+1\big)^{-\frac12}e^{-W_n}
\ll K_{c,n}^{-1}n^{\frac12}\big(\sfrac12-c+\sfrac1n\big)^{-\frac12}e^{-(\frac1{2e}-\delta)n}.
\end{align*}
Hence it follows from Remark \ref{Kremark} that also the second term in \eqref{klar2} is as small as desired. Finally, since the third term in \eqref{klar2} differs from the the third term in \eqref{klar1} only by a factor of polynomial size in $n$, the treatments of these terms are almost identical. Note in particular that the exponential decay in \eqref{decay} is uniform for $c\in[c_1,\frac12]$. This concludes the proof of the proposition.
\end{proof}

\begin{remark}\label{dualJ}
Recall from \eqref{Fn3} that we are interested in $J_n(L^*,cn)$. Since the measure $\mu_n$ is invariant under the homeomorphism $L\mapsto L^*$ of $X_n$ onto itself, we have the following consequence of Proposition \ref{JAprop}: Given any $c_1\in(\frac{1}{4},\frac12)$ there exists a constant $k>0$ such that 
\begin{align*}
\text{Prob}_{\mu_n}\Big\{L\in X_n\,\Big|\,\,K_{c,n}^{-1}\big|J_n(L^*,cn)\big|<e^{-kn},\,\forall c\in[c_1,\sfrac12]\Big\}\to1
\end{align*}
as $n\to\infty$.
\end{remark}

We collect the results of this section in the following theorem. 

\begin{thm}\label{EPSTEIN}
Let $c_1\in(\frac14,\frac12)$. Then for all $\ve>0$ there exists $A_0>0$ such that for all $A\geq A_0$ there exists $n_0\in\Z_{\geq3}$ such that for all $n\geq n_0$ we have 
\begin{align*}
\text{Prob}_{\mu_n}\bigg\{L\in X_n\,\Big|\,\sup_{c\in[c_1,\frac12)}\bigg|V_n^{-2c}E_n(L,cn)-\int_0^AV^{-2c}\,dR_n(V)\bigg|\leq\ve\bigg\}\geq1-\ve.
\end{align*}
\end{thm}

\begin{proof}
Recall that $V_n=\frac{\omega_n}{n}$. Since
\begin{align*}
K_{c,n}^{-1}F_n(L,cn)=V_n^{-2c}E_n(L,cn)
\end{align*}
the theorem follows from Proposition \ref{HAprop}, Lemma \ref{Htail} (cf.\ \eqref{important1}) and Remark \ref{dualJ}.
\end{proof}

\section{Proof of Theorem \ref{curvethm}}

Theorem \ref{EPSTEIN} says that for $c\in[c_1,\frac12)$ the random variable $\int_0^AV^{-2c}\,dR_n(V)$ is, with large probability, uniformly close to the (normalized) Epstein zeta function provided that $A$ and $n$ are appropriately large. We now show that this random variable is close in distribution to the corresponding truncation of $H(c)$. 

\begin{lem}\label{speciallemma}
Let $\frac14<c_1<c_2<\frac12$ and $A>0$ be fixed. Then the $C\big([c_1,c_2]\big)$-valued random function 
\begin{align*}
c\mapsto \int_0^AV^{-2c}\,dR_n(V)
\end{align*}
converges in distribution to the random function 
\begin{align*}
c\mapsto \int_0^AV^{-2c}\,dR(V)
\end{align*}
as $n\to\infty$.
\end{lem}

\begin{proof}
Expressed in more explicit terms, recalling the definitions of $R_n(V)$ and $R(V)$ (see \eqref{Rvdef} and \eqref{Rnvdef}), we need to prove that the random function $$c\mapsto 2\sum_{\mathcal V_j\leq A}\mathcal V_j^{-2c}-\frac{A^{1-2c}}{1-2c}$$ converges in distribution to $$c\mapsto 2\sum_{T_j\leq A}T_j^{-2c}-\frac{A^{1-2c}}{1-2c}$$ as $n\to\infty$. Note that the function $f_A$ defined in \eqref{FA}, considered as a function from $\Omega\setminus\Omega^{(\infty)}$ into $C\big([c_1,c_2]\big)$, is continuous on the open set $\cup_{j=0}^{\infty}(\Omega^{(j)})^{\circ}$ (cf.\ \eqref{mangd1}, \eqref{mangd2}), which has full ($\textbf P$-)measure in $\Omega$. Now the lemma follows from \cite[Thm.\ 1$'$]{poisson} and \cite[Thm.\ 2.7]{billconv}.
\end{proof}

We let $\mathcal P\big(C\big([c_1,c_2]\big)\big)$ denote the set of Borel probability measures on $C\big([c_1,c_2]\big)$. We recall that for $P,Q\in\mathcal P\big(C\big([c_1,c_2]\big)\big)$ the L\'{e}vy-Prohorov distance $\pi(P,Q)$ between $P$ and $Q$ is defined as 
\begin{align}\label{LP}
 \pi(P,Q):=\inf\Big\{\ve>0\,\big|\,P(B)\leq Q(B^{\ve})+\ve\,\text{ for all Borel sets $B\subseteq C\big([c_1,c_2]\big)$}\Big\}\,,
\end{align}
where $B^{\ve}$ is the open $\ve$-neighbourhood of $B$ in $C\big([c_1,c_2]\big)$ (cf.\ \cite{billconv}). Since $C\big([c_1,c_2]\big)$ is separable, it is known that convergence in the metric $\pi$ is equivalent to weak convergence in $\mathcal P\big(C\big([c_1,c_2]\big)\big)$. 

\begin{proof}[Proof of Theorem \ref{curvethm}]\label{curveproof}
 Let $\ve>0$ be given and let $\mu_{E_n}$, $\mu_{E_{n,A}}$, $\mu_{H_A}$ and $\mu_{H}$ be the distributions of the $C\big([c_1,c_2]\big)$-valued random functions $c\mapsto V_n^{-2c}E_n(\cdot,cn)$, $c\mapsto \int_0^AV^{-2c}\,dR_n(V)$, $c\mapsto \int_0^AV^{-2c}\,dR(V)$ and $c\mapsto H(c)$, respectively. Let further $A>0$ and $n_0\in\Z_{\geq3}$ be large enough for Theorem \ref{EPSTEIN}, Lemma \ref{speciallemma} and Lemma \ref{intest} to guarantee that $\pi(\mu_{E_n},\mu_{E_{n,A}})\leq\ve$, $\pi(\mu_{E_{n,A}},\mu_{H_A})\leq\ve$ and $\pi(\mu_{H_A},\mu_{H})\leq\ve$ hold for all $n\geq n_0$. It follows from the triangle inequality that $\pi(\mu_{E_n},\mu_{H})\leq3\ve$ for all $n\geq n_0$. We conclude that $\mu_{E_n}$ converges (in the metric $\pi$) to $\mu_{H}$ as $n\to\infty$ and the theorem follows.
 \end{proof}
 
\begin{remark}\label{cancellationremark}
We note that our claim in \eqref{EXPCANCELLATION} about exponential cancellation in $H_n(L,cn)$ follows easily from \eqref{important1} and Lemma \ref{speciallemma}. Indeed, given $\ve>0$ we choose $A>0$ and $n_0\in\Z_{\geq3}$ such that \eqref{important1} holds for all $n\geq n_0$, and using Lemma \ref{speciallemma} we see that there exists some $M>0$ and $n_0'\in\Z_{>0}$ such that for all $n\geq n_0'$ we have $\bigl|\int_0^A V^{-2c}\,dR_n(V)\bigr|<M$ for our fixed $c\in(\frac14,\frac12)$, with ($\mu_n$-)probability $\geq1-\ve$. It follows that 
\begin{align*}
\text{Prob}_{\mu_n}\Bigl\{L\in X_n\:\Big|\:\bigl|H_n(L,cn)\bigr|<(M+\ve)K_{c,n}\Bigr\}\geq1-2\ve
\end{align*}
for all $n\geq\max(n_0,n_0')$. But $K_{c,n}\ll(2c)^{cn}n^{-(c+\frac12)}$ as $n\to\infty$, and thus for all sufficiently large $n$ we have $(M+\ve)K_{c,n}<e^{-\delta n}$, where $\delta:=-c\log(2c)>0$. Since $\ve>0$ was arbitrary, this concludes the proof of \eqref{EXPCANCELLATION}.
\end{remark}

\section{An extension of Theorem \ref{curvethm} and  proofs of Theorem \ref{myheight} and Corollary \ref{negativity}}\label{heightsec}

In this section we are interested in extending the result in Theorem \ref{curvethm} to the case $c_2=\frac12$. The problem is that neither $E_n(L,cn)$ nor $H(c)$ is defined for $c=\frac 12$. We overcome this problem by subtracting the singular part of $E_n(L,cn)$ from $E_n(L,cn)$ and $H(c)$. For the rest of this section we let $c_1\in(\frac14,\frac12)$ be fixed.

Recall that $E_n(L,s)$ has a simple pole at $s=\frac n2$ with residue $\pi^{\frac n2}\Gamma(\frac n2)^{-1}$. Hence, for all $n$ and all $L\in X_n$, the limit 
\begin{align*}
\lim_{c\to\frac12}\bigg(E_n(L,cn)-\frac{\pi^{\frac n2}}{\Gamma(\frac n2)(cn-\frac n2)}\bigg)
\end{align*}
exists. Now, since 
\begin{align*}
\lim_{c\to\frac12}V_n^{-2c}
=\frac{n}{\omega_n}=\frac{n\Gamma(\frac n2)}{2\pi^{\frac n2}},
\end{align*}
basic complex analysis gives that also the limit
\begin{align}\label{epsteinlimit}
\lim_{c\to\frac12}\bigg(V_n^{-2c}E_n(L,cn)+\frac{1}{1-2c}\bigg)
\end{align}
exists for all $n$ and all $L\in X_n$. Hence 
we can consider 
\begin{align*}
c\mapsto \widehat E_n(\cdot,cn):=V_n^{-2c}E_n(\cdot,cn)+\frac{1}{1-2c}
\end{align*}
as a $C\big([c_1,\frac12]\big)$-valued random function. Here, of course, the value of the function at $c=\frac12$ is given by the limit \eqref{epsteinlimit}. We now have the following immediate corollary of Theorem \ref{EPSTEIN}.

\begin{cor}\label{EPSTEIN2}
Let $c_1\in(\frac14,\frac12)$. Then for all $\ve>0$ there exists $A_0>1$ such that for all $A\geq A_0$ there exists $n_0\in\Z_{\geq3}$ such that for all $n\geq n_0$ we have 
\begin{align*}
\text{Prob}_{\mu_n}\bigg\{L\in X_n\Big|\sup_{c\in[c_1,\frac12]}\bigg|\widehat E_n(L,cn)-\bigg(\int_0^1V^{-2c}\,dN_n(V)+\int_1^AV^{-2c}\,dR_n(V)\bigg)\bigg|\leq\ve\bigg\}\\
\geq1-\ve.
\end{align*}
\end{cor}

\begin{proof}
Note that 
\begin{align*}
\int_0^AV^{-2c}\,dR_n(V)+\frac{1}{1-2c}=\int_0^1V^{-2c}\,dN_n(V)+\int_1^AV^{-2c}\,dR_n(V)
\end{align*}
for all $c\in[c_1,\frac12)$. Hence the corollary follows from Theorem \ref{EPSTEIN} since both $\widehat E_n(L,cn)$ and $\int_0^1V^{-2c}\,dN_n(V)+\int_1^AV^{-2c}\,dR_n(V)$ are continuous on $[c_1,\frac12]$, for each fixed $L\in X_n$. 
\end{proof}


We set 
\begin{align*}
\widehat H(c):=\int_0^{\infty}V^{-2c}\,dR(V)+\frac{1}{1-2c}\,\,\text{ for $c\in[c_1,\sfrac12)$ and }\,\,\widehat H(\sfrac12):=Z_0.
\end{align*}
It follows from Lemma \ref{welldef1}, Lemma \ref{welldef3}, Lemma \ref{Z0lemma}, Remark \ref{Z0remark} and \cite[p.\ 84]{billconv} that we can consider $c\mapsto \widehat H(c)$ as a $C\big([c_1,\frac12]\big)$-valued random function on $\Omega$ (cf.\ Remark \ref{welldef2}). Furthermore we note that \eqref{goodrep} and Remark \ref{Z0formula} give, for all $c\in[c_1,\frac12]$, the formula
\begin{align*}
\widehat H(c)=\int_0^1V^{-2c}\,dN(V)+\int_1^{\infty}V^{-2c}\,dR(V).
\end{align*}
We are now ready to prove the following extension of Theorem \ref{curvethm}.

\begin{thm}\label{curvethm2}
Let $c_1\in(\frac14,\frac{1}{2})$. Then the distribution of the $C\big([c_1,\frac12]\big)$-valued random function $c\mapsto \widehat E_n(\cdot,cn)$ converges to the distribution of $c\mapsto \widehat H(c)$ as $n\to\infty$.
\end{thm}

\begin{proof}
Let $h_A\in C\big([c_1,\frac12]\big)$ be given by
\begin{align*}
h_A(c)=
\begin{cases}
\frac{1-A^{1-2c}}{1-2c}&\text{if $c\in[c_1,\frac12)$,}\\
-\log A&\text{if $c=\frac12$.}
\end{cases}
\end{align*}
To begin with we note that the function $g_A:\Omega\setminus\Omega^{(\infty)}\to C\big([c_1,\frac12]\big)$, defined by
\begin{align*}
g_A(x_1,x_2,\ldots)(c)=2\sum_{x_j\leq A}x_j^{-2c}+h_A(c),
\end{align*}
is continuous $\textbf P$ almost everywhere (cf.\ the proofs of Lemma \ref{welldef1} and Lemma \ref{speciallemma}). Hence it follows from \cite[Thm.\ 1$'$]{poisson} and \cite[Thm.\ 2.7]{billconv} that the $C\big([c_1,\frac12]\big)$-valued random function $$c\mapsto 2\sum_{\mathcal V_j\leq A}\mathcal V_j^{-2c}+h_A(c)=\int_0^1V^{-2c}\,dN_n(V)+\int_1^AV^{-2c}\,dR_n(V)$$ converges in distribution to $$c\mapsto 2\sum_{T_j\leq A}T_j^{-2c}+h_A(c)=\int_0^1V^{-2c}\,dN(V)+\int_1^{A}V^{-2c}\,dR(V)$$ as $n\to\infty$. The theorem now follows from this fact, Lemma \ref{intest} and Corollary \ref{EPSTEIN2} using the L\'{e}vy-Prohorov metric (see \eqref{LP}) in a way almost identical to the one in the proof of Theorem \ref{curvethm} on p.\ \pageref{curveproof}.
\end{proof}

\begin{proof}[Proof of Corollary \ref{negativity}]
Let $\frac14<c_1<c_2\leq\frac{1}{2}$ be given. To start with, we assume $c_2<\frac12$. Let $\mathcal C$ be the following open subset of $C\big([c_1,c_2]\big)$: 
\begin{align*}
\mathcal C:=\Big\{f\in C\big([c_1,c_2]\big)\,\big|\,f(c)<0\,\text{ for all }\, c\in[c_1,c_2]\Big\}.
\end{align*}
Note that
\begin{align*}
\partial\mathcal C=\Bigl\{f\in C\big([c_1,c_2]\big)\:\big|\:\sup_{c\in[c_1,c_2]}f(c)=0\Bigr\}.
\end{align*}
Let $\mu_H$ be the distribution of the $C\big([c_1,c_2]\big)$-valued random function $c\mapsto H(c)$. We claim that
\begin{align}\label{MYHPARTIAL0}
 \mu_H(\partial\mathcal C)=0. 
\end{align}
To prove this, recall that since $0<T_1<T_2<\ldots$ are the points of a Poisson process $\mathcal P$ on the positive real line with constant intensity $\frac12$, they can be realized as the partial sums of an infinite sequence of independent random variables which each has the exponential distribution with parameter $\frac12$ (cf.\ \cite[Sec.\ 4.1]{kingman}). It follows from this that if we parametrize $\Omega$ by the homeomorphism $J:\R_{>0}\times\R_{>0}\times\Omega\to\Omega$ given by $J(u,v,\vecz)=\vecx$ with $x_1=u$, $x_2=u+v$ and $x_j=u+v+z_{j-2}$ for $j\geq3$, then 
\begin{align*}
d\mathbf P(\vecx)=\frac14 e^{-\frac12u}e^{-\frac12v}\,du\,dv\,d\mathbf P(\vecz).
\end{align*}
Hence
\begin{align*}
\mu_H(\partial\mathcal C)=\frac14\int_\Omega\int_0^\infty\int_0^\infty I\bigl(J(u,v,\vecz)\in\mathcal S\bigr)
\, e^{-\frac12u}e^{-\frac12v}\,du\,dv\,d\mathbf P(\vecz),
\end{align*}
where 
\begin{align*}
\mathcal S=\Bigl\{\vecx\in\Omega\:\big|\:\sup_{c\in[c_1,c_2]}H(c)=0\Bigr\}.
\end{align*}
Substituting $v=y-u$ we get
\begin{align}\label{MYHPARTIAL0PF1}
\mu_H(\partial\mathcal C)=\frac14\int_\Omega\int_0^\infty\int_0^yI\bigl(J(u,y-u,\vecz)\in\mathcal S\bigr)
\, e^{-\frac12y}\,du\,dy\,d\mathbf P(\vecz).
\end{align}
Now for a given point $\vecx=J(u,y-u,\vecz)$ we have (assuming $\vecx\in\Omega_{1/2}$, or equivalently $\vecz\in\Omega_{1/2}$)
\begin{align}\label{HFIXKONST}
H(c)&=\int_0^\infty V^{-2c}\,dR(V)=\int_0^{x_2} V^{-2c}\,dR(V)+\int_{x_2}^{\infty} V^{-2c}\,dR(V)\\
&=2u^{-2c}+2y^{-2c}-\frac{y^{1-2c}}{1-2c}+\int_{x_2}^{\infty} V^{-2c}\,dR(V),\nonumber
\end{align}
where the last integral is independent of $u$ for given $y,\vecz$. Since in fact all terms in the second line of \eqref{HFIXKONST} except the first are independent of $u$, and $u^{-2c}$ is a decreasing function of $u>0$ for every fixed $c\in[c_1,c_2]$, it follows that if $J(u,y-u,\vecz)\in\mathcal S$ for some $0<u<y$ and $\vecz\in\Omega_{1/2}$, then $J(u',y-u',\vecz)\notin\mathcal S$ for all $u'$ with $0<u'<u$ or $u<u'<y$. Hence the innermost integral in \eqref{MYHPARTIAL0PF1} vanishes for all $y>0$ and $\vecz\in\Omega_{1/2}$, and we conclude that \eqref{MYHPARTIAL0} holds.

Using \eqref{MYHPARTIAL0}, the first part of the corollary now follows from Theorem \ref{curvethm} and \cite[Thm.\ 2.1]{billconv}.

In the remaining case $\frac14<c_1<c_2=\frac12$ we consider instead the open set
\begin{align*}
\widehat{\mathcal C}:=\biggl\{f\in C\big([c_1,\sfrac12]\big)\:\Big|\: f(c)-\frac1{1-2c}<0\text{ for all }c\in[c_1,\sfrac12)\biggr\},
\end{align*}
and let $\mu_{\widehat H}$ be the distribution of the $C\big([c_1,\frac12]\big)$-valued random function $c\mapsto\widehat H(c)$. Now 
\begin{align*}
\mu_{\widehat H}\big(\partial\widehat{\mathcal C}\,\big)=0
\end{align*}
holds, with almost the same proof as before. (Indeed, this boils down to proving that the triple integral in \eqref{MYHPARTIAL0PF1} vanishes, where now 
\begin{align*}
\mathcal S=\Big\{\vecx\in\Omega\:\big|\: \sup_{c\in[c_1,\frac12)} H(c)=0\Big\},
\end{align*}
and the same argument as before applies, since $\lim_{c\to\frac12-}H(c)=-\infty$ for all $\vecx\in\Omega_{1/2}$.)
Hence by Theorem \ref{curvethm2} and  \cite[Thm.\ 2.1]{billconv} we have
\begin{multline*}
\lim_{n\to\infty}\text{Prob}_{\mu_n}\biggl\{L\in X_n\:\Big|\:\widehat E_n(L,cn)-\frac1{1-2c}<0\text{ for all }c\in[c_1,\sfrac12)\biggr\}\\
=\text{Prob}\biggl\{\widehat H(c)-\frac1{1-2c}<0\text{ for all }c\in[c_1,\sfrac12)\biggr\}.
\end{multline*}
This implies that the first part of the corollary holds also when $c_2=\frac12$.

In order to prove $0<f(c_1,c_2)<1$ for general $\frac14<c_1<c_2\leq\frac12$, we let $\Omega(A)=\{\vecx\in\Omega\mid x_1>A\}$ for $A>0$. Clearly for any $\vecx\in\Omega(A)$ we have 
\begin{align*}
\int_0^AV^{-2c}\,dR(V)=-\int_0^AV^{-2c}\,dV=-\frac{A^{1-2c}}{1-2c}
\end{align*}
for all $c\in[c_1,c_2]\setminus\{\frac12\}$. Hence, by differentiation with respect to $c$, we find that for all $A>1$ and $\vecx\in\Omega(A)$ we have $\int_0^AV^{-2c}\,dR(V)\leq-e\log A<0$ for all $c\in[c_1,c_2]\setminus\{\frac12\}$. Recall from Lemma \ref{intest} that given $\ve>0$ there exists $A>1$ such that with probability $\geq1-\ve$ we have $\sup_{c\in[c_1,c_2]}\big|\int_{A}^{\infty}V^{-2c}\,dR(V)\big|\leq\ve$. Note that, since the Poisson process $\mathcal P$ may be realized as the superposition of a Poisson process on $(0,A)$ and an independent Poisson process on $(A,\infty)$, both with constant intensity $\frac12$ (cf., e.g., \cite[Sec.\ 2.2]{kingman}), and since furthermore $\int_A^{\infty}V^{-2c}\,dR(V)$ only depends on those points of $\mathcal P$ which belong to $(A,\infty)$, the probability of $\sup_{c\in[c_1,c_2]}\big|\int_{A}^{\infty}V^{-2c}\,dR(V)\big|\leq\ve$ remains unchanged if we condition on $\vecx\in\Omega(A)$. Hence, for small enough $\ve$ and large enough $A$, we have $f(c_1,c_2)\geq(1-\ve)\textbf P(\Omega(A))>0$. Finally, by a similar argument where we instead condition on the event $N(A)=B$ for some large $B$, we also obtain $f(c_1,c_2)\leq\textbf P\big\{\vecx\in\Omega\mid H(c_1)<0\big\}<1$.
\end{proof}

Theorem \ref{curvethm2} also has the following corollary.

\begin{cor}\label{lastcorollary}
The random variable 
\begin{align*}
\widehat E_n\big(\cdot,\sfrac n2\big)=\lim_{c\to\frac12}\bigg(V_n^{-2c}E_n(\cdot,cn)+\frac{1}{1-2c}\bigg)
\end{align*}
converges in distribution to $Z_0$ as $n\to\infty$. 
\end{cor}

\begin{proof}
Given $c_1\in(\frac14,\frac12)$ the evaluation map $C\big([c_1,\frac12]\big)\ni f\mapsto f(\frac12)$ is continuous. Hence the desired result follows from Theorem \ref{curvethm2} and  \cite[Thm.\ 2.7]{billconv}. 
\end{proof}

As a consequence of this result we obtain an easy proof of Theorem \ref{myheight}.

\begin{proof}[Proof of Theorem \ref{myheight}]
First, applying the functional equation \eqref{functionaleq} and \eqref{Fn}, we get
\begin{align}\label{newfunctionaleq}
E_n\big(L,\sfrac n2-s\big)&=\pi^{\frac n2-s}\Gamma\big(\sfrac n2-s\big)^{-1}F_n\big(L,\sfrac n2-s\big)
=\pi^{\frac n2-s}\Gamma\big(\sfrac n2-s\big)^{-1}F_n\big(L^*,s\big)\\
&=\pi^{\frac n2-2s}\Gamma\big(\sfrac n2-s\big)^{-1}\Gamma(s)E_n\big(L^*,s\big).\nonumber
\end{align} 
We are interested in this relation when $s$ is small. Using \eqref{heightdef} and basic knowledge about the gamma function we have, for $s$ sufficiently small,
\begin{align*}
&\pi^{\frac n2-2s}=\pi^{\frac n2}\Big(1-2(\log\pi)s+O\big(s^2\big)\Big);\\
&\Gamma\big(\sfrac n2-s\big)^{-1}=\Big(\Gamma(\sfrac n2)-\Gamma'(\sfrac n2)s+O\big(s^2\big)\Big)^{-1}\\
&\hspace{55pt}=\Gamma(\sfrac n2)^{-1}\bigg(1-\frac{\Gamma'(\sfrac n2)}{\Gamma(\sfrac n2)}s+O\big(s^2\big)\bigg)^{-1}=\Gamma(\sfrac n2)^{-1}\bigg(1+\frac{\Gamma'(\sfrac n2)}{\Gamma(\sfrac n2)}s+O\big(s^2\big)\bigg);\\
&\Gamma(s)=s^{-1}\Gamma(s+1)=s^{-1}-\gamma+O(s);\\
&E_n\big(L^*,s\big)=-\Big(1-\big(h_n(L)-2\log(2\pi)\big)s+O\big(s^2\big)\Big),
\end{align*}
where $\gamma$ is Euler's constant and the implied constants are allowed to depend on $n$. Using these expansions in \eqref{newfunctionaleq} yields
\begin{align}\label{svariant}
E_n\big(L,\sfrac n2-s\big)&=-\pi^{\frac n2}\Gamma(\sfrac n2)^{-1}\bigg(s^{-1}+\Big(2\log2+\frac{\Gamma'(\sfrac n2)}{\Gamma(\sfrac n2)}-h_n(L)-\gamma\Big)+O(s)\bigg).
\end{align}
Writing $\frac n2-s$ as $cn$ and using the relation $\pi^{\frac n2}\Gamma(\sfrac n2)^{-1}=\frac12\omega_n$ we get, for $|c-\sfrac12|$ sufficiently small,
\begin{align*}
E_n\big(L,cn\big)&=-\frac{\omega_n}{n}\bigg(\frac{1}{1-2c}+\frac n2\Big(2\log2+\frac{\Gamma'(\sfrac n2)}{\Gamma(\sfrac n2)}-h_n(L)-\gamma\Big)+O\big(|c-\sfrac12|\big)\bigg).
\end{align*}
Since we furthermore have
\begin{align*}
V_n^{-2c}=\frac{n}{\omega_n}\Big(\frac{\omega_n}{n}\Big)^{1-2c}=\frac{n}{\omega_n}\Big(1+\big(\log\omega_n-\log n\big)(1-2c)+O\big(|c-\sfrac12|^2\big)\Big),
\end{align*}
we obtain
\begin{align*}
&V_n^{-2c}E_n\big(L,cn\big)\\
&=-\bigg(\frac{1}{1-2c}+\log\omega_n-\log n+\frac n2\Big(2\log2+\frac{\Gamma'(\sfrac n2)}{\Gamma(\sfrac n2)}-h_n(L)-\gamma\Big)+O\big(|c-\sfrac12|\big)\bigg).
\end{align*}
Hence, we conclude that
\begin{align}\label{limitexpression}
\lim_{c\to\frac12}\bigg(V_n^{-2c}E_n(L,cn)+\frac{1}{1-2c}\bigg)=\log n-\log\omega_n+\frac n2\bigg(h_n(L)+\gamma-2\log2-\frac{\Gamma'(\sfrac n2)}{\Gamma(\sfrac n2)}\bigg).
\end{align}

Next we study the asymptotics of \eqref{limitexpression} as $n\to\infty$. Using \eqref{omegan} and Stirling's formula we get
\begin{multline*}
\lim_{c\to\frac12}\bigg(V_n^{-2c}E_n(L,cn)+\frac{1}{1-2c}\bigg)
=\log n-\frac n2\log\Big(\frac{2\pi e}{n}\Big)-\frac12\log\Big(\frac{n}{\pi}\Big)\\
+o(1)+\frac n2\Big(h_n(L)+\gamma-2\log2-\log\Big(\frac n2\Big)+n^{-1}+O\big(n^{-2}\big)\Big)\\
=\frac12\log(\pi n)+\frac12+\frac n2\Big(h_n(L)-\big(\log(4\pi)-\gamma+1\big)\Big)+o(1).
\end{multline*}
Hence we conclude that
\begin{align*}
2\widehat E_n\big(L,\sfrac n2\big)-\log\pi-1=n\Big(h_n(L)-\big(\log(4\pi)-\gamma+1\big)\Big)+\log n+o(1),
\end{align*}
where $o(1)$ stands for a certain function of $n$ which is independent of $L$ and which tends to $0$ as $n\to\infty$. Using e.g.\ the L\'evy-Prohorov metric on $\mathcal P(\R)$ (the set of Borel measures on $\R$), it now follows from Corollary \ref{lastcorollary} that
\begin{align*}
n\Big(h_n(L)-\big(\log(4\pi)-\gamma+1\big)\Big)+\log n
\end{align*}
converges in distribution to $2Z_0-\log\pi-1$ as $n\to\infty$, which is the desired result.
\end{proof}

\begin{remark}
Our proof shows that Theorem \ref{myheight} is really a special case of Theorem \ref{curvethm2}, and we think this nicely illustrates the power of Theorem \ref{curvethm2}. However, it is worth noticing that considerations involving $C\big([c_1,\frac12]\big)$-valued random functions are not at all essential for the proof of Theorem \ref{myheight}: An alternative proof of Theorem \ref{myheight} can be given by working more directly along the lines of Sarnak and Str\"ombergsson \cite[Sec.\ 6]{sst} and applying the $R_n(V)$-bound in Theorem \ref{Rnthm} and our Poisson limit result from \cite{poisson}.

To outline this alternative approach, recall from \cite[Sec.\ 4]{sst} that
\begin{align}\label{lastheight}
h_n(L)=\log(4\pi)-\gamma-\frac2n+{\sum_{\vecm\in L^*}}'G\big(0,\pi |\vecm|^2\big)+{\sum_{\vecm\in L}}'G\big(\sfrac n2,\pi |\vecm|^2\big),
\end{align}
where we call the two sums above $J(L)$ and $H(L)$ respectively. Using the same notation as in Section \ref{Epsteinsection} we have, for any $A>0$, 
\begin{multline}\label{HLEXPANSION}
H(L)=\int_0^{\infty}G\Big(\frac n2,\pi\Big(\frac{nV}{\omega_n}\Big)^{\frac{2}{n}}\Big)\,dN_n(V)
=\int_0^{A}G\Big(\frac n2,\pi\Big(\frac{nV}{\omega_n}\Big)^{\frac{2}{n}}\Big)\,dN_n(V)\\
+\int_A^{\infty}G\Big(\frac n2,\pi\Big(\frac{nV}{\omega_n}\Big)^{\frac{2}{n}}\Big)\,dV
+\int_A^{\infty}G\Big(\frac n2,\pi\Big(\frac{nV}{\omega_n}\Big)^{\frac{2}{n}}\Big)\,dR_n(V).
\end{multline}
The last integral in \eqref{HLEXPANSION} can be bounded using Theorem \ref{Rnthm} and Lemma \ref{sstbound}. (The computations are exactly as in the proof of Lemma \ref{Htail} but a tiny bit simpler as we are
working only with $c=\frac12$, instead of aiming at a uniform bound over the interval $c\in[c_1,\frac12]$.) The result is that the random variable
\begin{align*}
n\int_A^\infty G\Bigl(\frac n2,\pi\Bigl(\frac{nV}{\omega_n}\Bigr)^{\frac2n}\Bigr)\,dR_n(V)
\end{align*}
converges in distribution to the constant $0$, as $A,n\to\infty$. (Naturally, this also follows as a consequence of Lemma \ref{Htail}, since $K_{1/2,n}=\frac 2n$.) Regarding the first integral in the right hand side of \eqref{HLEXPANSION}, the same argument as in Proposition \ref{HAprop} (cf.\ also \eqref{firstA} and Lemma \ref{GAMMA}) shows that, for fixed $A>0$, the distributions of the two random variables
\begin{align*}
n\int_0^AG\Bigl(\frac n2,\pi\Bigl(\frac{nV}{\omega_n}\Bigr)^{\frac2n}\biggr)\,dN_n(V)\qquad\text{and}\qquad2\int_0^A V^{-1}\,dN_n(V)
\end{align*}
have L\'evy-Prohorov distance tending to $0$ as $n\to\infty$. Furthermore, applying \cite[Thm.\ 1$'$]{poisson} and \cite[Thm.\ 2.7]{billconv} in the usual way (this time for real-valued random variables), it follows that the random variable $2\int_0^A V^{-1}\,dN_n(V)$ converges in distribution to $2\int_0^A V^{-1}\,dN(V)$ as $n\to\infty$. Finally, the middle integral in the right hand side of \eqref{HLEXPANSION} can be evaluated asymptotically as $n\to\infty$, for example as follows. Using Lemma \ref{int} and Lemma \ref{GAMMA} we find that, for fixed $A>0$ and with an arbitrary fixed constant $0<\delta<\frac1{2e}$,
\begin{align*}
\int_A^\infty G\Bigl(\frac n2,\pi\Bigl(\frac{nV}{\omega_n}\Bigr)^{\frac2n}\Bigr)\,dV&=\lim_{s\to\frac n2-}\biggl(\frac1{\frac n2-s}-\int_0^A G\Bigl(s,\pi\Bigl(\frac{nV}{\omega_n}\Bigr)^{\frac2n}\Bigr)\,dV\biggr)\\
&=\biggl(\lim_{s\to\frac n2-}\frac{1-\Gamma(\frac n2)^{-1}\Gamma(s)\bigl(\pi(\frac{nA}{\omega_n})^{2/n}\bigr)^{\frac n2-s}}{\frac n2-s}\biggr)+O\big(e^{-\delta n}\big)\\
&=\frac{\Gamma'(\frac n2)}{\Gamma(\frac n2)}-\log\Bigl(\pi\Bigl(\frac{nA}{\omega_n}\Bigr)^{2/n}\Bigr)+O\big(e^{-\delta n}\big)\\
&=1-n^{-1}\log n-n^{-1}\bigl(1+\log\pi+2\log A+o(1)\bigr)
\end{align*}
as $n\to\infty$. Collecting these results, and also using the fact that the random variable $2\int_0^AV^{-1}\,dN(V)-2\log A$ converges in distribution to $2Z_0$ as $A\to\infty$, we conclude that the random variable $n\big(H(L)-1\big)+\log n$ converges in distribution to $2Z_0-\log\pi-1$ as $n\to\infty$.

Similarly,
\begin{multline*}
J(L^*)=\int_0^{\infty}G\Big(0,\pi\Big(\frac{nV}{\omega_n}\Big)^{\frac{2}{n}}\Big)\,dN_n(V)\\
=\int_0^{\infty}G\Big(0,\pi\Big(\frac{nV}{\omega_n}\Big)^{\frac{2}{n}}\Big)\,dV
+\int_0^{\infty}G\Big(0,\pi\Big(\frac{nV}{\omega_n}\Big)^{\frac{2}{n}}\Big)\,dR_n(V).
\end{multline*}
Here the first integral in the right hand side can be evaluated explicitly by Lemma \ref{int}, and equals $\mathbb E\big(J(L^*)\big)=\frac2n$. The second integral in the right hand side can be bounded using Theorem \ref{Rnthm} and Lemma \ref{sstbound} in the same way as in the proof of Proposition \ref{JAprop}; the result is that the random variable
\begin{align*}
n\int_0^\infty G\Bigl(0,\pi\Bigl(\frac{nV}{\omega_n}\Bigr)^{\frac2n}\Bigr)\,dR_n(V)
\end{align*}
converges in distribution to the constant $0$, as $n\to\infty$. (The same result also follows as a consequence of Proposition \ref{JAprop}, for $c=\frac12$.) Hence we conclude that the random variable $nJ(L^*)$ (and hence also $nJ(L)$) converges in distribution to the constant $2$ as $n\to\infty$. Theorem \ref{myheight} now follows from \eqref{lastheight} and the above limit results for $n\big(H(L)-1\big)+\log n$ and $nJ(L^*)$.
\hfill$\square$
\end{remark}

\section{Proof of Theorem \ref{independentgaussians}}

In this section we prove Theorem \ref{independentgaussians}. The proof is based on a study of the joint moments of an explicit truncation of  
\begin{align*}
\Big(\big(2c_1-\sfrac12\big)^{\frac12}H(c_1),\ldots,\big(2c_m-\sfrac12\big)^{\frac12}H(c_m)\Big).
\end{align*} 
To be more precise we will, for $\delta>0$, consider the random vector
\begin{align}\label{deltatruncation}
\Big(\big(2c_1-\sfrac12\big)^{\frac12}H(c_1,\delta),\ldots,\big(2c_m-\sfrac12\big)^{\frac12}H(c_m,\delta)\Big),
\end{align} 
where
\begin{align*}
H(c,\delta):=\int_{\delta}^{\infty}V^{-2c}\,dR(V).
\end{align*}
In order to calculate the joint moments of the random vector \eqref{deltatruncation} we first prove a formula closely related to \cite[Prop.\ 3]{poisson}.

\begin{prop}\label{newlem}
Let $k\geq1$ and denote by $\mathcal P'(k)$ the set of partitions of $\{1,\ldots,k\}$ containing no singleton sets. For $1\leq j\leq k$ let $f_j:\R_{\geq0}\to\R$ be functions satisfying $\prod_{j\in B}f_j\in L^1(\R_{\geq0})$ for every nonempty subset $B\subseteq\{1,\ldots,k\}$. Then 
\begin{align*}
\mathbb E\bigg(\prod_{j=1}^k\int_{0}^{\infty}f_j(V)\,dR(V)\bigg)=\underset{P \in\mathcal{P}'(k)}{\sum}2^{k-\#P}\underset{B\in P}{\prod}\bigg(\int_{0}^{\infty}\prod_{j\in B}f_j(V)\,dV\bigg). 
\end{align*}
\end{prop}

\begin{remark}
In particular, when $1\leq k\leq 3$ Proposition \ref{newlem} gives
\begin{align*}
&\mathbb E\bigg(\int_{0}^{\infty}f_1(V)\,dR(V)\bigg)=0\,;\\
&\mathbb E\bigg(\prod_{j=1}^2\int_{0}^{\infty}f_j(V)\,dR(V)\bigg)=2\int_{0}^{\infty}f_1(V)f_2(V)\,dV\,;\\
&\mathbb E\bigg(\prod_{j=1}^3\int_{0}^{\infty}f_j(V)\,dR(V)\bigg)=4\int_{0}^{\infty}f_1(V)f_2(V)f_3(V)\,dV.
\end{align*}
\end{remark}

\begin{proof}[Proof of Proposition \ref{newlem}]
Let $K=\{1,\ldots,k\}$. Note that for each $1\leq j\leq k$ we have 
\begin{align*}
\int_{0}^{\infty}f_j(V)\,dR(V)=2\sum_{n=1}^{\infty}f_j(T_n)-\int_0^{\infty}f_j(V)\,dV.
\end{align*}
Using this observation together with \cite[Prop.\ 3]{poisson} we get
\begin{align}\label{firstnewlem}
&\mathbb E\bigg(\prod_{j=1}^k\int_{0}^{\infty}f_j(V)\,dR(V)\bigg)\\
&=\sum_{A\subset K}(-1)^{\#(K\setminus A)}\bigg(\prod_{j\in K\setminus A}\int_0^{\infty}f_j(V)\,dV\bigg)\mathbb E\bigg(\prod_{j\in A}2\sum_{n=1}^{\infty}f_j(T_n)\bigg)\nonumber\\
&=\sum_{A\subset K}\sum_{P\in\mathcal P(A)}(-1)^{\#(K\setminus A)}2^{\#A-\#P}\bigg(\prod_{j\in K\setminus A}\int_0^{\infty}f_j(V)\,dV\bigg)\prod_{B\in P}\int_0^{\infty}\prod_{j\in B}f_j(V)\,dV,\nonumber
\end{align}
where $\mathcal{P}(A)$ denotes the set of partitions of the set $A$. Given $A\subset K$ and $P\in\mathcal P(A)$ we define $P'(A,P)$ to be the partition
\begin{align*}
P'(A,P):=\big\{\{j\}\mid j\in K\setminus A\big\}\cup P
\end{align*}
of $K$. Rewriting the right hand side of \eqref{firstnewlem} in terms of partitions of $K$ yields
\begin{align}\label{secondnewlem}
\sum_{A\subset K}\sum_{P\in\mathcal P(A)}(-1)^{\#(K\setminus A)}2^{\#A-\#P}\prod_{B\in P'(A,P)}\int_0^{\infty}\prod_{j\in B}f_j(V)\,dV.
\end{align}

For each partition $P'\in\mathcal P(K)$ we let $S(P')\subset K$ denote the union of the singleton sets in $P'$. Note that in the double sum \eqref{secondnewlem} we have $P'(A,P)=P'$ for exactly $2^{\# S(P')}$ pairs $(A,P)$. Indeed, when $K\setminus A$ runs through all subsets of $S(P')$ there exists, for each such $A$, a unique partition $P\in\mathcal P(A)$ with $P'(A,P)=P'$. We conclude that \eqref{secondnewlem} equals
\begin{align*}
&\sum_{P'\in\mathcal P(K)}2^{k-\#P'}\bigg(\sum_{C\subset S(P')}(-1)^{\#C}\bigg)\prod_{B\in P'}\int_0^{\infty}\prod_{j\in B}f_j(V)\,dV\\
&=\sum_{\substack{P'\in\mathcal P(K)\\S(P')=\emptyset}}2^{k-\#P'}\prod_{B\in P'}\int_0^{\infty}\prod_{j\in B}f_j(V)\,dV,
\end{align*}
which is the desired result.
\end{proof}

We note that the functions $$g_{c,\delta}(V):=V^{-2c}I(V\geq\delta)$$ do not satisfy the assumption in Proposition \ref{newlem} for any choice of $c\in(\frac14,\frac12)$ and $\delta>0$. However, by an approximation argument we get the following corollary.

\begin{cor}\label{newcor}
Let $k\geq1$ and let $\delta>0$ and $\frac14<c_1\leq\ldots\leq c_k<\frac12$ be fixed. Then
\begin{align*}
\mathbb E\bigg(\prod_{j=1}^kH(c_j,\delta)\bigg)=\underset{P \in\mathcal{P}'(k)}{\sum}2^{k-\#P}\delta^{\#P-2\sum_{j=1}^{k}c_j}\underset{B\in P}{\prod}\frac{1}{2\sum_{j\in B}c_j-1}. 
\end{align*}
\end{cor}

\begin{proof}
For all $\frac14<c<\frac12$ and $A>\delta$ we let
\begin{align*}
f_A(c,\delta):=\int_{\delta}^{A}V^{-2c}\,dR(V)=2\sum_{\delta<T_j\leq A}T_j^{-2c}-\frac{A^{1-2c}-\delta^{1-2c}}{1-2c}.
\end{align*}
As in Lemma \ref{welldef1} we find that $f_A(c,\delta)$ is a measurable function on $\Omega$. We furthermore note that Lemma \ref{intest} implies that the random vector $$\big(f_A(c_1,\delta),\ldots,f_A(c_k,\delta)\big)$$ tends in distribution to $$\big(H(c_1,\delta),\ldots,H(c_k,\delta)\big)$$ as $A\to\infty$. Hence it follows from \cite[Thm.\ 2.7]{billconv} that $\prod_{j=1}^kf_A(c_j,\delta)$ converges in distribution to  $\prod_{j=1}^kH(c_j,\delta)$ as $A\to\infty$.

Now, applying Proposition \ref{newlem}, we get
\begin{align*}
\mathbb E\bigg(\Big|\prod_{j=1}^kf_A(c_j,\delta)\Big|^2\bigg)=\underset{P \in\mathcal{P}'(2k)}{\sum}2^{2k-\#P}\underset{B\in P}{\prod}\int_{\delta}^{A}\prod_{j\in B}V^{-2\tilde c_j}\,dV,
\end{align*}
where $\tilde c_{2j}=\tilde c_{2j-1}=c_j$ for $1\leq j\leq k$. By the dominated convergence theorem, together with the fact that $\#B\geq2$ for all $B\in P\in\mathcal P'(2k)$, we have 
\begin{align*}
\lim_{A\to\infty}\mathbb E\bigg(\Big|\prod_{j=1}^kf_A(c_j,\delta)\Big|^2\bigg)=\underset{P \in\mathcal{P}'(2k)}{\sum}2^{2k-\#P}\underset{B\in P}{\prod}\int_{\delta}^{\infty}\prod_{j\in B}V^{-2\tilde c_j}\,dV
\end{align*}
and hence, in particular, it follows that $\sup_{A>\delta}\mathbb E\bigg(\Big|\prod_{j=1}^kf_A(c_j,\delta)\Big|^2\bigg)<\infty$. Similarly we find that
\begin{align*}
\lim_{A\to\infty}\mathbb E\bigg(\prod_{j=1}^kf_A(c_j,\delta)\bigg)=\underset{P \in\mathcal{P}'(k)}{\sum}2^{k-\#P}\underset{B\in P}{\prod}\int_{\delta}^{\infty}\prod_{j\in B}V^{-2c_j}\,dV
\end{align*}
and the corollary now follows from \cite[Cor.\ to Thm.\ 25.12]{billing}.
\end{proof}

In the special case where $c_1=\ldots=c_k=c$, Corollary \ref{newcor} gives
\begin{align}\label{normal1}
\mathbb E\big(H(c,\delta)^k\big)=\underset{P \in\mathcal{P}'(k)}{\sum}2^{k-\#P}\delta^{\#P-2kc}\underset{B\in P}{\prod}\frac{1}{2c\#B-1}. 
\end{align}
In the next lemma we will consider the rescaled variable 
\begin{align}\label{finalvariable}
\mathscr H(c,\delta):=\big(2c-\sfrac12\big)^{\frac12}\delta^{2c-\frac12}H(c,\delta),
\end{align}
which by \eqref{normal1} satisfies $\mathbb E\big(\mathscr H(c,\delta)\big)=0$, $\mathbb E\big(\mathscr H(c,\delta)^2\big)=1$ and 
\begin{align}\label{normal2}
\mathbb E\big(\mathscr H(c,\delta)^k\big)=\big(2c-\sfrac12\big)^{\frac k2}\underset{P \in\mathcal{P}'(k)}{\sum}2^{k-\#P}\delta^{\#P-\frac k2}\underset{B\in P}{\prod}\frac{1}{2c\#B-1},\qquad k\geq3. 
\end{align}
If $k\geq 3$ is odd, then for every $P\in\mathcal P'(k)$ we have $\#\{B\in P\mid\#B=2\}\leq\frac12(k-3)$. Hence it follows from \eqref{normal2} that, for fixed $\delta>0$ and odd $k\geq3$, we have $\lim_{c\to\frac14+}\mathbb E\big(\mathscr H(c,\delta)^k\big)=0$. Similarly we find that, for fixed $\delta>0$ and even $k\geq4$, we have
\begin{align*}
\lim_{c\to\frac14+}\mathbb E\big(\mathscr H(c,\delta)^k\big)&=\#\big\{P\in\mathcal P'(k)\mid\#B=2,\; \forall B\in P\big\}\\
&=\binom{k}{2,\ldots,2}\frac{1}{(k/2)!}=(k-1)!!.
\end{align*}
Since these limits coincide with the corresponding moments of the distribution $N(0,1)$ and normal distributions are determined by their moments, we conclude that, for any fixed $\delta>0$, $\mathscr H(c,\delta)$ converges in distribution to $N(0,1)$ as $c\to\frac14+$. More generally, we have the following result.

\begin{lem}\label{normallemma}
Fix $m\in\Z_{\geq1}$ and let $c_j=\frac14+\eta_j$ with $\eta_j\in(0,\frac14)$ for $1\leq j\leq m$. If $\delta>0$ is fixed and $(\eta_1,\ldots,\eta_m)$ tends to the zero vector in $\R^m$ in such a way that $\eta_j/\eta_{j+1}\to0$ for each $1\leq j\leq m-1$, then the $m$-dimensional random vector $\big(\mathscr H(c_1,\delta),\ldots,\mathscr H(c_m,\delta)\big)$ converges in distribution to the distribution of $m$ independent $N(0,1)$-variables.
\end{lem}

\begin{proof}
It remains to consider the case where $m\geq2$. Let $k_1,\ldots,k_m\in\Z_{\geq0}$ satisfying $k=k_1+\ldots+k_m\geq1$ be given and let
\begin{align*}
\tilde c_j=
\begin{cases}
c_1 & \text{if $1\leq j\leq k_1$},\\
c_2 &  \text{if $k_1<j\leq k_1+k_2$},\\
\hspace{2pt}\vdots &\\
c_m &  \text{if $k_1+\ldots+k_{m-1}<j\leq k$}.
\end{cases}
\end{align*}
It follows from Corollary \ref{newcor} and \eqref{finalvariable} that
\begin{align}\label{normal3}
&\mathbb E\bigg(\prod_{j=1}^m\mathscr H(c_j,\delta)^{k_j}\bigg)\\
&=\bigg(\prod_{j=1}^{m}(2\eta_j)^{\frac{k_j}{2}}\bigg)\delta^{2\sum_{j=1}^{m}k_j\eta_j}\underset{P \in\mathcal{P}'(k)}{\sum}2^{k-\#P}\delta^{\#P-2\sum_{j=1}^{m}k_jc_j}\underset{B\in P}{\prod}\frac{1}{2\sum_{j\in B}\tilde c_j-1}.\nonumber
\end{align}

In this sum, the contribution from a given partition $P\in\mathcal P'(k)$ is, in the limit under consideration, writing $M_{i_1,i_2}$ for the number of elements $B\in P$ which satisfy $\#B=2$, $\min B\in(k_1+\ldots+k_{i_1-1},k_1+\ldots+k_{i_1}]$ and $\max B\in(k_1+\ldots+k_{i_2-1},k_1+\ldots+k_{i_2}]$ (here $k_1+\ldots+k_0:=0$), 
\begin{align}\label{normal4}
&\asymp\prod_{j=1}^{m}\eta_j^{\frac{k_j}{2}}\underset{\substack{B\in P\\\#B=2}}{\prod}\frac{1}{2\sum_{j\in B}\tilde c_j-1}\\
&\asymp\prod_{j=1}^{m}\eta_j^{\frac{k_j}{2}}\underset{1\leq i_1\leq i_2\leq m}{\prod}(\eta_{i_1}+\eta_{i_2})^{-M_{i_1,i_2}}
\asymp\prod_{j=1}^{m}\eta_j^{\frac{k_j}{2}}\underset{1\leq i_1\leq i_2\leq m}{\prod}\eta_{i_2}^{-M_{i_1,i_2}}\nonumber\\
&=\bigg(\prod_{j=1}^{m-1}\Big(\frac{\eta_j}{\eta_{j+1}}\Big)^{\frac12\sum_{\ell=1}^jk_{\ell}-\sum_{1\leq i_1\leq i_2\leq j}M_{i_1,i_2}}\bigg)\eta_m^{\frac12\sum_{\ell=1}^mk_{\ell}-\sum_{1\leq i_1\leq i_2\leq m}M_{i_1,i_2}}\nonumber.
\end{align} 
Hence, since by definition we have $\sum_{1\leq i_1\leq i_2\leq j}M_{i_1,i_2}\leq\frac12\sum_{\ell=1}^jk_{\ell}$ for each $1\leq j\leq m$, the expression in \eqref{normal4} tends to zero unless 
\begin{align}\label{normal5}
\sum_{1\leq i_1\leq i_2\leq j}M_{i_1,i_2}=\frac12\sum_{\ell=1}^jk_{\ell},\qquad\forall j\in\{1,\ldots,m\}.
\end{align}

Now suppose that $P\in\mathcal P'(k)$ gives a non-zero limit contribution to \eqref{normal3}. From \eqref{normal5} we get $M_{1,1}=\frac{1}{2}k_1$, which implies that $M_{1,j}=0$ for all $2\leq j\leq m$. Next \eqref{normal5} gives $M_{1,1}+M_{1,2}+M_{2,2}=\frac12(k_1+k_2)$. By our previous observations we must have $M_{2,2}=\frac{1}{2}k_2$ and hence it follows that $M_{2,j}=0$ for all $3\leq j\leq m$. Continuing in the same way we find that \eqref{normal5} forces $M_{j,j}=\frac{1}{2}k_j$ for all $1\leq j\leq m$ and $M_{i,j}=0$ whenever $i<j$. Conversely, we note that these conditions imply that \eqref{normal5} holds. Thus, in particular, the moment in \eqref{normal3} tends to zero unless all $k_j$ are even. Furthermore, for each partition  $P\in\mathcal P'(k)$ satisfying \eqref{normal5} the contribution to \eqref{normal3} equals
\begin{align*}
&\bigg(\prod_{j=1}^{m}(2\eta_j)^{\frac{k_j}{2}}\bigg)2^{k-\#P}\delta^{\#P-2\sum_{j=1}^{m}k_j(c_j-\eta_j)}\underset{B\in P}{\prod}\frac{1}{2\sum_{j\in B}\tilde c_j-1}=1.
\end{align*}
Hence, in the limit under consideration, the moment $\mathbb E\big(\prod_{j=1}^m\mathscr H(c_j,\delta)^{k_j}\big)$ tends to the number of partitions $P\in\mathcal P'(k)$ satisfying condition \eqref{normal5}. Recalling the discussion below \eqref{normal2} we conclude that
\begin{align*}
\mathbb E\bigg(\prod_{j=1}^m\mathscr H(c_j,\delta)^{k_j}\bigg)\to\prod_{j=1}^mM_{k_j},
\end{align*}
where 
\begin{align*}
M_k:=\begin{cases}
0 & \text{ if $k$ is odd,}\\
(k-1)!! & \text{ if $k$ is even.}
\end{cases}
\end{align*}
The lemma follows since a random vector whose coordinates in the standard basis are independent $N(0,1)$-variables is determined by its joint moments.
\end{proof}

\begin{proof}[Proof of Theorem \ref{independentgaussians}]
Let $c_j=\frac14+\eta_j$ with $\eta_j\in(0,\frac14)$ for $1\leq j\leq m$. For convenience of  notation we set $\mathscr H(c):=\big(2c-\sfrac12\big)^{\frac12}H(c)$. We note that if $\delta>0$ is fixed and $\vecx\in\Omega$ is such that $N(\delta)=0$, then 
\begin{align}\label{normal6}
\mathscr H(c)=\big(2c-\sfrac12\big)^{\frac12}\bigg(H(c,\delta)-\int_0^{\delta}V^{-2c}\,dV\bigg)=\delta^{\frac12-2c}\mathscr H(c,\delta)-\big(2c-\sfrac12\big)^{\frac12}\frac{\delta^{1-2c}}{1-2c},
\end{align}
where $\lim_{c\to\frac14+}\delta^{\frac12-2c}=1$ and $\lim_{c\to\frac14+}\big(2c-\sfrac12\big)^{\frac12}\frac{\delta^{1-2c}}{1-2c}=0$.

Now let $\ve>0$ be given. Fix $\delta>0$ small enough to ensure that $\textbf{P}(N(\delta)=0)>1-\ve$. By Lemma \ref{normallemma} there exist numbers $0<\tilde\eta_1<\ldots<\tilde\eta_m<\frac14$ such that for all vectors $(\eta_1,\ldots,\eta_m)$ satisfying $0<\eta_j\leq\tilde\eta_j$ ($1\leq j\leq m$) as well as $0<\frac{\eta_j}{\eta_{j+1}}<\frac{\tilde\eta_j}{\tilde\eta_{j+1}}$ ($1\leq j\leq m-1$), the distribution of $\big(\mathscr H(c_1,\delta),\ldots,\mathscr H(c_m,\delta)\big)$ is within $\ve$ of the distribution of $m$ independent $N(0,1)$-variables in the L\'{e}vy-Prohorov metric. Furthermore it follows from \eqref{normal6} that we can, by possibly shrinking the numbers $\tilde\eta_j$,  guarantee that 
\begin{align*}
\textbf{P}\bigg(\Big|\big(\mathscr H(c_1),\ldots,\mathscr H(c_m)\big)-\big(\mathscr H(c_1,\delta),\ldots,\mathscr H(c_m,\delta)\big)\Big|<\ve\bigg)>1-2\ve,
\end{align*}
for all admissible $(\eta_1,\ldots,\eta_m)$. The observations above together imply that the distribution of $\big(\mathscr H(c_1),\ldots,\mathscr H(c_m)\big)$ is, for all admissible $(\eta_1,\ldots,\eta_m)$, within $3\ve$ of the distribution of $m$ independent $N(0,1)$-variables in the L\'{e}vy-Prohorov metric. This concludes the proof.
\end{proof}

\subsubsection*{Acknowledgement} I am most grateful to my advisor Andreas Str\"ombergsson for suggesting that the main theorems in the present paper should be possible to prove and for many helpful and inspiring discussions on this work. I would also like to thank Svante Janson for bringing to my attention the fact that $H(c)$ and $Z_0$ have stable distributions.

\end{document}